\documentclass[11pt, letterpaper]{amsart}
\usepackage[left=1in,right=1in,bottom=1in,top=1in]{geometry}
\usepackage{graphicx}
\usepackage{amsmath,amsthm,amssymb}
\usepackage{mathrsfs}
\usepackage[all,cmtip]{xy}
\usepackage{cite}
\usepackage{setspace}
\usepackage{youngtab}
\usepackage{young}

\theoremstyle{plain}
\newtheorem{lem}{Lemma}
\newtheorem{thm}[lem]{Theorem}
\newtheorem{prop}[lem]{Proposition}
\newtheorem{cor}[lem]{Corollary}

\theoremstyle{definition}
\newtheorem{defn}[lem]{Definition}

\theoremstyle{remark}

\newtheorem{example}[lem]{Example}

\onehalfspace

\begin{document}

\title{Likelihood Orders for some Random Walks on the Symmetric Group}
\author{Megan Bernstein}
\begin{abstract}
A random walk converging to its stationary distribution admits an order on its states from most to least likely. Here Fourier analysis and representations of $S_n$ are used to find the order after sufficient time and an upper bound for that time for several random walks on the symmetric group: the transposition walk, three-cycle walk, and $n$-cycle walk. This method can aid in finding the total variation distance and separation distance for a random walk on a group.\end{abstract}

\maketitle

\section{Introduction}

Random walks on the symmetric group have been a testing ground for methods of ascertaining the mixing of Markov chains ever since Diaconis and Shashahani showed cutoff for the random transposition walk ~\cite{DS}.   Two motivating notions of distance between a random walk on a group after $t$ steps and its stationary distribution, $\pi$, are separation distance and total variation distance defined as: \[ \text{sep}(t) = \max_{g \in G} \frac{ \pi(g) - P^{*t}(g)}{\pi(g)}\] \[ ||P^{*t} - \pi||_{TV} = \sum_{g \in G: P^{*t}(g) > \pi(g)} P^{*t}(g) - \pi(g) \] Since random walks on groups have uniform stationary distributions, the separation distance is attained at the least likely element.  The related $l^\infty$ distance is attained at either the most or least likely element.  For total variation, useful bounds, especially lower bounds, frequently originate from understanding the likelihood of the elements relative to the uniform distribution.  It is then natural and of interest to know the most and least likely elements of a random walk as well as the likelihood order.

A short argument from Diaconis and Isaacs ~\cite{DI} shows that at even times the identity element of a group is the most likely element of any symmetric random walk on a group.  The only other known method to determine the relative likelihood of the elements is to develop a likelihood order for all elements of the group.  Diaconis and Isaacs explore several walks, including on the cyclic group, with orders that hold at all times using induction. This paper maintains the need for a complete order but extends to orders that do not hold at all times.

Here Fourier analysis will be used to find likelihood orders that hold after sufficient time for random walks generated by transpositions, $3$-cycles, and $n$-cycles.  The classification of representations as $i$-cycle detectors will yield for these walks the order defining representation.  This leads to several cycle lexicographical likelihood orders. The multiplicity and spectral gap to the next largest character ratio of a representation will generate a heuristic for sufficient time.  

For the transposition walk, this paper partially answers a conjecture of Isaacs and Diaconis that for the lazy transposition walk the $n$-cycles are always the least likely elements.  This method proves it after order $n^2$ steps.  It is trivially true up to $n$ steps. However, the likelihood order, as a whole, breaks frequently for $t<n$.  Whether the $n$-cycles are the least likely and the likelihood order holds in this gap are open.

Section \ref{B} outlines the techniques, discrete Fourier analysis and character theory for the symmetric group. Section \ref{icyc} establishes the $i$-cycle detectors as the representations that determine the likelihood orders. Following that, a brief foray in section \ref{orders} into the cycle lexicographical orders that will appear as likelihood orders and a comparison with other partial orders on partitions. Then for each of the following walks, transposition in section \ref{transp}, three-cycle in section \ref{three}, and $n$-cycle in section \ref{ncycle}, the likelihood order and a bound for sufficient time are found. Following the transposition walk is a extension to lazy walks, section \ref{lazy}, and a method for finding the states more likely than the stationary distribution, section \ref{uniform}. 

\section{Background}\label{B}

\subsection{Notation}

The letters $\lambda, \rho, \gamma$ will always refer to partitions.  $\alpha, \beta, \kappa$ to conjugacy classes of $S_n$, $\alpha =[\alpha_1,...,\alpha_r] = (1^{a_1},...,n^{a_n})$ with cycles lengths $\alpha_1 \geq \cdots \geq \alpha_r$ with $a_i$ $i$-cycles. The following partitions occur repeatedly and will be denoted by $[n-i,i]= \lambda^i$,$[n-i,i-k,1,...,1]=\lambda^{i,k}$, and $[n-i,1,...,1]=\rho^i=\lambda^{i,i-1}$.

Since these random walks are generated by conjugacy classes the probability function is a class function. This means that probabilities are equal within a conjugacy class. Formulas will be written in terms of conjugacy classes referring to probability of an individual element of the conjugacy class.

\subsection{Discrete Fourier Inversion Formula}\label{DFIF}

The Fourier inversion formula gives an expression for the distribution of a random walk on the symmetric group in terms of characters of irreducible representations of the symmetric group.  These irreducibles are indexed by partitions of $n$. See Diaconis ~\cite{Diaconis} for a more thorough treatment.

\begin{prop}\label{fformula}
For a walk starting at the identity with first step $P(\cdot)$ a class function, the $t^{th}$ step is given by
\[ P^{*t}(\alpha) = \frac{1}{n!}\sum_{\lambda} \chi_{\lambda}(\alpha) d_{\lambda} (C)^t \]
Where $C$ is as follows. The sum below is over conjugacy classes $\kappa$ of size $|\kappa|$, \[C = \sum_{\kappa } |\kappa|P(\kappa)\frac{\chi_{\lambda}(\kappa)}{d_\lambda}\] 
\end{prop}

\begin{prop}\label{bigformula}
\[ P^{*t}(\alpha) - P^{*t}(\beta) = \frac{1}{n!}\sum_{\lambda} \left( \chi_{\lambda}(\alpha) - \chi_{\lambda}(\beta) \right) d_{\lambda} (C)^t \]
Where $C$ is as follows. The sum below is over conjugacy classes $\kappa$ of size $|\kappa|$, \[C = \sum_{\kappa } |\kappa|P(\kappa)\frac{\chi_{\lambda}(\kappa)}{d_\lambda}\] 
\end{prop}

\subsection{Murnaghan-Nakayama}\label{CT}

\begin{prop}
\[ \chi_{\lambda}(\alpha) = \sum_S (-1)^{\text{ht}(S)} \]
summed over all sequences of partition $S = (\lambda^{(0)}, \lambda^{(1)},...,\lambda^{(r)})$ such that $r = l(\alpha)$, $0 = \lambda^{(0)} \subset \lambda^{(1)} \subset \cdots \subset \lambda^{(r)}$ such that each $\lambda^{(i)} - \lambda^{(i-1)}$ is a border strip of length $\alpha_i$ and $\text{ht}(S) = \sum_i \text{ht}(\lambda^{(i)} - \lambda^{(i-1)})$.
\end{prop}

One way of viewing this process is sucessively removing border strips of length $\alpha_r,..., \alpha_1$ in all possible ways from the bottom, right of $\lambda$.  Alternatively, one can envision this process as in Littlewood, ~\cite{Littlewood}, as successive insertions of $\alpha_1,...,\alpha_r$ into the top,left of $\lambda$.  This reversal of the usual visualization was key in defining an $i$-cycle detector as it emphasizes the importance of the large pieces.  The following are borrowed from Littlewood with some change in terminology.

\begin{defn}
The insertion of $i$ nodes to a partition is called a valid insertion of $i$ nodes if the nodes are added to any row until they are exhausted or until the number of nodes in this row exceeds the number in the preceding row by one, the nodes being then added to the preceding row according to the same rule, and so on until the $i$ nodes are exhausted provided the final product is a valid partition.  If the number of rows involved is even it is called a negative application, if odd, a positive application.
\end{defn}

\begin{prop}
If $\lambda$ is a partition of $n$ and $\alpha$ denotes a conjugacy class of the symmetric group with cycles of orders $\alpha_1,...,\alpha_r$ the $\chi_\lambda(\alpha)$ is obtained form the number of methods of building the partition $\lambda$ by consectutive valid insertions of $\alpha_1,...,\alpha_r$ nodes by subtracting the number of ways which contain an odd number of negative applications from the number of ways which contain an even number of negative applications.
\end{prop}

An insertion of $\alpha_i$ nodes from a cycle length $\alpha_i$ will be shortened to an insertion of an $\alpha_i$ cycle.

\begin{example}
To see this, consider calculating $\chi_{[4,2]}(1^2,4)$. First to choose which row to start inserting the largest cycle, the four cycle, into the shape $[4,2]$:  
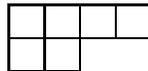
\[\begin{Young} &&&\cr&\cr\end{Young}\]
Either the the first row or the second row are possible giving (the nodes are denumerated by the $i$ of the $\alpha_i$ that fills it):
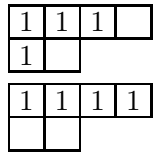
\[\begin{Young} 1 & 1 & 1  &\cr 1 & \cr \end{Young}\]
\[\begin{Young} 1 & 1 & 1 & 1 \cr & \cr \end{Young}\]
The first is a negative insertion since it covers an even number of rows, while the second confined to the first row is a positive insertion. It remains to place the two $1$-cycles. In the first, the first $1$-cycle can go into either the first or second row to be valid, and the second $1$-cycle must go in the remaining spot.
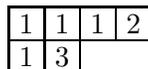
\[\begin{Young} 1 & 1 & 1  & 2\cr 1 &3 \cr \end{Young}\]
\[\begin{Young} 1 & 1 & 1  & 3\cr 1 &2 \cr \end{Young}\]
While our second way of placing the $4$-cycle leaves only the second row for each $1$-cycle insertion. 
\[\begin{Young} 1 & 1 & 1 & 1 \cr 2 & 3 \cr \end{Young}\]
So this sums to two ways with an odd number of negative insertions and one way with an even number of negative insertions. This gives $\chi_{[4,2]}(1^2,4)= -2 +1=-1$

\end{example}

\subsection{Character Polynomials}
Another useful tool for insight into the characters is the character polynomial, see ~\cite{CP}.  The most well known character polynomial is the one for $[n-1,1]$, that $\chi_{[n-1,1]}(\alpha) = a_1-1$, so that the number of fixed points completely determines this character. 
\begin{defn} 
The character polynomial of $\mu$ a partition of $n$ is \[q_\mu(x_1,...,x_{n}) = \downarrow \left( \sum\limits_{\alpha \dashv n} \frac{\chi_{\mu}(\alpha)}{z_\alpha} \prod\limits_{i=1}^{n} (ix_i - 1)^{a_i} \right) \] where $z_\alpha = \prod_i{a_i! i^{a_i}}$ and $\downarrow(x_1^{a_1} \cdots x_n^{a_n}) = (x_1)_{a_1} \cdots (x_n)_{a_n}$.
\end{defn}

\begin{prop} 
\[\chi_{\lambda}(\alpha) = q_{[\lambda_2,...,\lambda_r]}(a_1,...,a_{n-\lambda_1}) \]
\end{prop}

\begin{example}
 To continue using $\chi_{[4,2]}(1^2,4)$ as an example, find the character polynomial corresponding to $\chi_{[4,2]}(\alpha)$. This corresponds to removing the first row of the partition, so $q_{[2]}$.
\[q_{[2]}(x_1,x_2) = \downarrow \left( \sum\limits_{\alpha \dashv 2} \frac{\chi_{[2]}(\alpha)}{z_\alpha} \prod\limits_{i=1}^{2} (ix_i - 1)^{a_i} \right) \] 
Since $\chi_{[2]}(\cdot)=1$ and the only partitions of $2$ are $[2]$ and $[1]$ with $z_{[2]} = (1!)(2)$,$z_{[1,1]}=(2!)(1^2)=2$, this gives,
\[q_{[2]}(x_1,x_2) = \downarrow \left( \frac{1}{2}(2x_2 - 1) + \frac{1}{2}(x_1 - 1)^{2}\right) = x_2 -\frac{1}{2} + \frac{1}{2} (\downarrow x_1^2) -x_1 + \frac{1}{2} = x_2 + {{x_1 \choose 2}} -x_1 \] 
Applying this to $(1^2,4)$ gives $\chi_{[4,2]}(1^2,4)= 0 + {{2 \choose 2}} - 2 = -1$. 

\end{example}

\section{Detecting Cycle Structure}\label{icyc}

As motivated by the formula for the difference in probabilities, the goal here is to describe, for fixed conjugacy classes $\alpha,\beta$, partitions, $\lambda$, for which we know $\chi_\lambda(\alpha) - \chi_\lambda(\beta) =0$.  Each the character for each partition has a granularity to detect cycle structure up to a size beyond which it is indescriminant.  For example, above it was noted that $\chi_{[n-1,1]}$ is determined by fixed points and it was computed that $\chi_{[n-2,2]}$ sees only fixed points and $2$-cycles. This leads to three equivalent conditions motivated by both Murnaghan-Nakayama and character polynomials characterizing a partition with the potential to see $i$-cycles, to be called an $i$-cycle detector.  In turn, a partition that is not an $i$-cycle detector will not be able to detect if $\alpha,\beta$ only differ in cycle decompositions for cycles $\geq i$.  The definitions below reflect the property that $\chi_\lambda(\alpha) = \text{sgn}(\alpha)\chi_{\lambda'}(\alpha)$, so being an $i$-cycle dector is a dual statement about a partition and its conjugate.

\begin{defn} 
An insertion of cycles lengths $\alpha_1,...,\alpha_k$ into $\lambda$ is trival if they insert with nodes in the same order as from inserting a cycle length $\alpha_1 + ...+ \alpha_k$. Note that this requires $\alpha_1$ to occupy the entire first column of the hook of $\alpha_1+...+\alpha_k$.
\end{defn}

\begin{example}
Recall the examples of insertions above.
\[\begin{Young} 1 & 1 & 1  & 2\cr 1 &3 \cr \end{Young}\]
\[\begin{Young} 1 & 1 & 1  & 3\cr 1 &2 \cr \end{Young}\] 
\[\begin{Young} 1 & 1 & 1 & 1 \cr 2 & 3 \cr \end{Young}\]
As always, the first cycle, $\alpha_1$ inserts trivially in all three examples. Now, looking at the first two cycles insertions, only the first example is trivial since in this case the second cycle was inserted following the first cycle in the same row. The second example fails to be trivial is there is no was to insert one cycle length $\alpha_1+\alpha_2=5$ into the shape $\alpha_1,\alpha_2$ fill, $[3,2]$. And the third fails since the first cycle does not fill the entire first column. The three cycles insert trivially in none of these examples.
\end{example}

\begin{defn} 
Call $\lambda$ a $i$-cycle detector if there is a non-trivial insertion of cycles lengths $\geq i$ into $\lambda$ and $\lambda'$
\end{defn}

Implicit in this definition is that $i$-cycle detectors only exist for $i \leq \frac{n}{2}$, since it is impossible to insert two cycles size $> \frac{n}{2}$ into a partition of $n$.

\begin{lem}\label{ILem}
The following are equivalent:
\begin{enumerate}
\item $\lambda$ is an $i$-cycle detector
\item $h_{2,1},h_{1,2} \geq i$ (where $h_{x,y}$ is the length of the hook starting at $(x,y)$ in $\lambda$)
\item some $x_j$ for $j \geq i$ occurs in a monomial with non-zero coefficient in both $q_{[\lambda_2,...,\lambda_r]}(x),q_{[\lambda'_2,...,\lambda'_{r'}]}(x)$
\end{enumerate}
\end{lem}
\begin{proof}

The equivalence of the first two statements will mostly be a proof by diagram.  The first two rows and columns of $\lambda$ take on one of five shapes where the captial letters stand for any number of boxes, and the lower case a single box.  These capital letters and $h_{1,2},h_{2,1}$ will be used abusively to stand for both the boxes they represent and the number of boxes they represent.

\[ \young(xD,C), \young(xyD,wv,C), \young(xyAzD,wvAu,C), \young(xyD,wv,BB,rs,C), \young(xyAzD,wvAu,BB,rs,C)\]

First to show if $h_{2,1},h_{1,2} \geq i$, then two cycles can be inserted into the first two rows and columns of $\lambda$ with the second inserting non-trivially.  Then it will also clearly work for $\lambda'$ as it has the same property.

For \young(xD,C), $D = h_{1,2} \geq i$, $C = h_{2,1} \geq i$ so inserting the first $i$-cycle into the first row and the second into $C$ is possible.

For \young(xyD,wv,C), $D +2 = h_{1,2} \geq i$, $C+2 = h_{2,1} \geq i$, so inserting the first $i$-cycle into the first row and the second into $h_{2,1}$ is possible.

For  \young(xyAzD,wvAu,C), $h_{1,2} = 3 + A +D \geq i$, $h_{2,1} = 3+ A + C \geq i$, so the first $i$-cycle fits in the first row and the second in the $h_{2,1}$.

For \young(xyD,wv,BB,rs,C), $h_{1,2} = 3 + B +D \geq i$, $h_{2,1} = 3 + B +C \geq i$.  The first column then is the same length as $h_{2,1}$, so insert the first $i$-cycle vertically.  And the second can insert non-trivially into $h_{1,2}$.

Finally, for \young(xyAzD,wvAu,BB,rs,C).  Insert trivially into the first row and column along \young(xyAzD,w,B), which is at least $i$ long since this is the same length as $h_{1,2}$.  Insert non-trivially into the remaining squares, which are the same length as $h_{2,1}$.

If two cycles of lengths at least $i$ can be inserted with one non-trivial into both $\lambda, \lambda'$, it needs to be shown that $h_{1,2},h_{2,1} \geq i$.   These same five shapes also describe the shape these insertions can fill within $\lambda$.  For one shape only, it can only be shown that inserting into $\lambda$ gives $h_{2,1} \geq i$.  The insertion into $\lambda'$ will then satisfy $h_{1,2} \geq i$.  Let the insertion into $\lambda$ be one of the shapes:

For \young(xD,C), since the second insertion is assumed to be non-trivial, it must be contained in $C$.  Therefore, $h_{2,1} \geq C \geq i$. To establish that $h_{1,2} \geq i$ use that $\lambda'$ satisfies the non-trivial requirement as well.

Then, \young(xyD,wv,C).  The non-trivial intersection cannot be in both $C$ and $D$ since this would leave only one square for the trivial insertion.  This means the non-trivial insertion is contained in either $h_{2,1}$ or $h_{1,2}$, leaving the size of the other for the trivial insertion.  So both $h_{2,1}$ and $h_{1,2}$ are at least $i$.

And for \young(xyAzD,wvAu,C).  If the non-trivial insertion intersects either $w,C$ or $z,D$ it cannot intersect the other since this would make that insertion longer than the preceding trivial insertion.  This leaves either at most $D + A + 3$ boxes or $C + A + 3$ boxes for the trivial insertion, each contained in $h_{1,2}$ and $h_{2,1}$.  The remaining space then for a single insertion is at most $C + A + 3$ boxes and $D + A + 3$ boxes.  So both hooks are at least $i$.  If instead the non-trivial insertion did not intersect $w,C,z,D$, it must fit into $v,A,u$ which is smaller than either hook, and both hooks are also at least $i$ in length. 

Then, \young(xyD,wv,BB,rs,C) as the conjugate of the previous shape follows exactly the same argument with relabeling.

Lastly,  \young(xyAzD,wvAu,BB,rs,C).  The non-trivial intersection must contain the hook $s,B,v,A,u$.  It may additionally contain either $z,D$ or $r,C$ but not both, as that would make it longer than the trivial insertion.  First if it contains $r,C$.  This makes the non-trivial in $C + 4 + B + A$ boxes, which is less than the length of $h_{2,1}$.  This leaves $D + 4 + B + A$ boxes the trivial fills, less than the length of $h_{2,1}$.  If the non-trivial fills $z,D$, then similarly it cannot also touch $r,C$, and must be in $D + 4 + B + A$ boxes, less than the length of $h_{1,2}$.  Again, leaving $C + 4 + A + B \leq h_{2,1}$ for the trivial insertion.  In the final case where the trivial is just $s,B,v,A,u$, $ i \leq A + B + 3$ this is less than $h_{2,1} \geq A + B + D + 4$ and $h_{1,2} \geq A + B + C + 4$.

The equivalence of the latter two statements in the theorem follows from expanding the character polynomial.
\[q_{\lambda_2,...,\lambda_r}(x) = \downarrow \left( \sum\limits_{\alpha \dashv n- \lambda_1} \frac{\chi_{\mu}(\alpha)}{z_\alpha} \prod\limits_{i=1}^{n - \lambda_1} (ix_i - 1)^{a_i} \right) \]
The sum is over $\alpha \dashv n- \lambda_1$ but when $\chi_{\mu}(\alpha) = 0$ the $\alpha$ term is $0$.  By Murnaghan-Nakayama, the largest cycle inserts first, and the largest cycle that can insert into $[\lambda_2,...,\lambda_r]$ is $h_{2,1}$.  The sum can then be restricted to $\alpha$ with parts of size at most $h_{2,1}$. An $x_i$ occurs in the $\alpha$ term only when $\alpha$ has a part of size $i$.  So no $x_i$ term can than occur in $q_{[\lambda_2,...,\lambda_r]}$ for $i > h_{2,1}$.  Next, to show $x_{h_{2,1}}$ occurs with non-zero coefficient.  Taking the sum over $\alpha$ with $\alpha_1 = h_{2,1}$,

\begin{align*}&\downarrow \left( \sum_{\alpha\dashv [n- \lambda_1] : \alpha_1 = h_{2,1}, i \neq 1, \alpha_i < h_{2,1}} \frac{\chi_{[\lambda_2,...,\lambda_r]}(\alpha)}{z_\alpha} \prod\limits_{i=1}^{n - \lambda_1} (a_ix_i - 1)^{a_i} \right) \\
&=  \downarrow \left( \sum_{\beta\dashv n- \lambda_1 -h_{2,1}: \beta_i < h_{2,1}} \frac{(-1)^{r+1}\chi_{[\lambda_3-1,...,\lambda_r-1]}(\beta)}{h_{2,1}z_\beta} (h_{2,1}x_{h_{2,1}}-1)\prod\limits_{i=1}^{h_{2,1}-1} (ix_i - 1)^{b_i} \right) \\
&= (-1)^{r+1}(x_{h_{2,1}} -\frac{1}{h_{2,1}}) \downarrow \left( \sum_{\beta\dashv n- \lambda_1 -h_{2,1}} \frac{\chi_{[\lambda_3-1,...,\lambda_r-1]}(\beta)}{z_\beta}\prod\limits_{i=1}^{n- \lambda_1 -h_{2,1}} (ix_i - 1)^{b_i} \right) \\
&=  (-1)^{r+1}(x_{h_{2,1}} -\frac{1}{h_{2,1}})q_{[\lambda_3-1,...,\lambda_r-1]} \end{align*}
No character polynomial can be zero since no character is zero, so the $x_{h_{2,1}}$ will have non-zero coefficient in $q_{[\lambda_2,...,\lambda_r]}$.  Similarly the $x_{h_{1,2}}$ term will have non-zero coefficient in $q_{[\lambda'_2,...,\lambda'_{r'}]}$

\end{proof}
Given a partition, the hook starting at $(2,1)$ will be called its subhook, and its length, $h_{2,1}$ the subhook length of the partition.  When $h_{2,1} \leq h_{1,2}$, the partition is a $i$-cycle detector if its subhook length is at least $i$.  The proof above also shows:

\begin{cor}\label{Icor} For a partition $\lambda$,

\begin{itemize}
\item $h_{2,1}$ is the largest $i$ for which $x_i$ occurs in $q_{[\lambda_2,...,\lambda_r]}(x)$
\item $h_{1,2}$ is the largest $i$ for which $x_i$ occurs in $q_{[\lambda'_2,...,\lambda'_{r'}]}(x)$
\end{itemize}
\end{cor}

\begin{thm}
If $\alpha$, $\beta$ with the same sign have $a_j = b_j$ for all $j<i$ and $\lambda$ is not a $i$-cycle detector, then $\chi_{\lambda}(\alpha) - \chi_\lambda(\beta) = 0$.
\end{thm}
\begin{proof}
If $\lambda$ is not an $i$-cycle detector, then one of $h_{2,1},h_{1,2} <i$ by lemma \ref{ILem}.  By corollary \ref{Icor}, this means one of $q_{[\lambda_2,...,\lambda_r]}(x), q_{[\lambda'_2,...,\lambda'_{r'}]}(x)$ has no $x_j$ terms for $j \geq i$.  If it is the character polynomial for $\lambda$, then \[\chi_{\lambda}(\alpha) =q_{[\lambda_2,...,\lambda_r]}(a_1,...,a_{i-1}) = q_{[\lambda_2,...,\lambda_r]}(b_1,...,b_{i-1}) = \chi_{\lambda}(\beta)\] If it is the character polynomial for $\lambda'$, \begin{align*} \chi_{\lambda}(\alpha) &= \text{sgn}(\alpha)\chi_{\lambda'}(\alpha) =\text{sgn}(\alpha)q_{[\lambda'_2,...,\lambda'_{r'}]}(a_1,...,a_{i-1}) \\&= \text{sgn}(\beta)q_{[\lambda'_2,...,\lambda'_{r'}]}(b_1,...,b_{i-1}) = \text{sgn}(\beta)\chi_{\lambda'}(\beta) = \chi_{\lambda}(\beta)\end{align*}
\end{proof}

The stronger statement that any $i$-cycle detector can differentiate between two conjugacy classes that first differ in their number of $i$-cycles is not true.  For instance, the character polynomial associated with $[n-3,2,1]$ has no degree two terms despite being a $2$-cycle detector.  This means, for example, $\chi_{[n-3,2,1]}(1^{n-4},2^2) = \chi_{[n-3,2,1]}(1^{n-4},4)$.  
 
\section{The Cycle Lexicographic Orders and Other Orders on Partitions}\label{orders}

Three total orders on partitions will appear as ordering from most to least likely elements from the three walks in this paper.  All of these are variants on the basic cycle lexicographic order, CL, on conjugacy classes. 

\begin{defn} Let $\alpha=(1^{a_1},2^{a_2},...,n^{a_n}), \beta = (1^{b_1},2^{b_2},...,n^{b_n})$ define $\alpha >_{CL} \beta$ when for $\min_k(a_k \neq b_k) = i$, $a_i > b_i$. $\alpha =_{CL} \beta$ exactly when $\alpha = \beta$.
\end{defn}
Throughout the paper, $i$ will be used to mean this first differing cycle size for any pair $\alpha,\beta$.
CL is distinct from the traditional orders on partitions: majorization/domination/natural, reverse lexicographical, and Lulov's lexicographical. Where lexicographical without the cycle prefix here refers to the $\alpha_i$'s rather than the $a_i$ in the notation $\alpha = [\alpha_1,...,\alpha_r] = (1^{a_1},...,n^{a_n})$.

\begin{defn} Majorization is defined as $\alpha \trianglerighteq \beta$ if for all $i$, $\sum\limits_{j \leq i} \alpha_j \geq \sum\limits_{j \leq i} \beta_j$.   Equivalently, $\alpha \trianglerighteq \beta$ if boxes in the Ferrers diagram of $\beta$ can be moved up and to the right to get the Ferrers diagram of $\alpha$.
\end{defn}

\begin{defn} In reverse lexicographical order $\alpha \geq_{RL} \beta$ if for the minimum $i$ that $\alpha_i \neq \beta_i$, $\alpha_i > \beta_i$. ~\cite{Stanley1}
\end{defn}

This is a refinement of majorization into a total order.

In ~\cite{Lulov} this next definition is called reverse lexicographic, which clashes with the canonical definition used above, ~\cite{Stanley1}, ~\cite{MacDonald}.  It flips the order of the $\lambda_i$.  Lulov mistakenly equated this order to CL.
\begin{defn} In  Lulov's lexicographical order $\alpha \geq_{L} \beta$ if the the maximum $i$ such that $\alpha_i \neq \beta_i$, $\alpha_i < \beta_i$.
\end{defn}

\begin{prop}
CL order is not linear extension of majorization order (and automatically also incompatible with reverse lexicographical).  Cycle lexicographical is also distinct from Lulov's lexicographical.  
\end{prop} 
\begin{proof}
The order under majorization and cycle lexicographical of the following partitions of $6$  are incompatible: $[5,1]$, $[4,2]$, $[3,1,1,1]$.  In majorization order, \[ [5,1] \trianglerighteq [4,2] \trianglerighteq [3,1,1,1] \] as one can see how to move boxes down and to the left to get the next shape \[\yng(5,1) \trianglerighteq \yng(4,2) \trianglerighteq \yng(3,1,1,1) \] while under cycle lexicographic order the order is neither the same or reversed.  In each case, $i=1$ and the partitions are ordered by number of fixed points: \[ [3,1,1,1] \geq_{CL} [5,1] \geq_{CL} [4,2] \]

The order under Lulov's lexicographical and cycle lexicographical of the following partitions of $6$ are incompatible: $[5,1]$, $[2,2,2]$, $[3,1,1,1]$
\[ \yng(5,1),\yng(2,2,2),\yng(3,1,1,1) \]
In Lulov's lexicographic order these are ordered by number of parts, \[[5,1] \geq_{L} [2,2,2] \geq_{L} [3,1,1,1] \] while in cycle lexicographic the order of the three is again determined by number of fixed points: \[ [3,1,1,1] \geq_{CL} [5,1] \geq_{CL} [2,2,2] \]   
\end{proof}

It is also the case that when CL order is taken on the conjugate of $\alpha$ and majorization or Lulov's lexicographical is taken on $\alpha$, there is still incompatibility.  For the first take $[5,1],[4,2],[4,1,1]$ and the second again $[5,1]$,$[2,2,2]$,$[3,1,1,1]$. 

The variants of cycle lexicographical that arise in this paper as likelihood orders are as follows.  Fix $\alpha=(1^{a_1},2^{a_2},...,n^{a_n}), \beta = (1^{b_1},2^{b_2},...,n^{b_n})$.

\begin{defn} Define $\alpha >_{-CL} \beta$ when for $\min_k(a_k \neq b_k) = i$, $a_i < b_i$. $\alpha =_{-CL} \beta$ exactly when $\alpha = \beta$.
\end{defn}

Note this is just the reversal of $\geq_{CL}$.

\begin{defn} Define $\alpha >_{(-1)^{i+1}CL} \beta$ when for $\min_k(a_k \neq b_k) = i$, for $i$ even $a_i < b_i$ or for $i$ odd $a_i > b_i$. $\alpha =_{(-1)^{i+1}CL} \beta$ exactly when $\alpha = \beta$.
\end{defn}
The largest and smallest elements under these orders will be the most and least likely elements of walks.  For reference later, based on the divisibility of $n$ and restricted to permutations with odd sign, $|_{-}$ or even sign, $|_{+}$, the largest, $\hat{1}$, and smallest, $\hat{0}$, elements of each of these orders are:

\begin{prop}\label{CL orders}
When $n$ is odd:

\[\begin{array}{c|ccc} &  \text{CL} & \text{-CL} & (-1)^{i+1}\text{CL} \\ \hline
\hat{1} & (1^n) & (n) & (1^n) \\
\hat{1}_{+} & 1^n) & (n) & (1^n) \\
\hat{1}_{-} & (1^{n-2},2) & (\frac{n-1}{2},\frac{n-1}{2}) & (1^{n-2},2) \\
\hat{0} & (n) & (1^n) & (2^{\frac{n-3}{2}},3) \\
\hat{0}_{+} & (n) & (1^n) & \text{ if } 4|(n+1),   (2^{\frac{n-3}{2}},3) \text{ else, }  (2^{\frac{n-5}{2}},5) \\
\hat{0}_{-} & (\frac{n-1}{2},\frac{n-1}{2}) & (1^{n-2},2) & \text{ if } 4|(n+1),   (2^{\frac{n-5}{2}},5) \text{ else, }  (2^{\frac{n-3}{2}},3) \end{array} \]

When $n$ is even:

\[\begin{array}{c|ccc} & \text{CL} & \text{-CL} & (-1)^{i+1}\text{CL} \\ \hline
\hat{1} & (1^n) & (n) & (1^n) \\
\hat{1}_{+} & (1^n) & (\frac{n}{2}^2)& (1^n) \\
\hat{1}_{-} & (1^{n-2},2^1) & (n)  & (1^{n-2},2) \\
\hat{0} & (n) & (1^n) & (2^{\frac{n}{2}}) \\
\hat{0}_{+} &  (\frac{n}{2}^2)& (1^n) & \text{ if } 4|(n),   (2^{\frac{n}{2}}) \text{ else, }  (2^{\frac{n-4}{2}},4) \\
\hat{0}_{-} & (n) & (1^{n-2},2) & \text{ if } 4|(n),   (2^{\frac{n-4}{2}},4) \text{ else, }  (2^{\frac{n}{2}}) \end{array} \]

\end{prop}
\begin{proof}

First, CL order.  When $\alpha = (1^n)$, $\beta \neq \alpha$, $\min_k(a_k \neq b_k) = 1$ and $a_1 = n > n-2 \geq b_1$.  $(1^n)$ is an even conjugacy class, so it is also the largest even element in CL order.  Among odd conjugacy classes, the transposition $(1^{n-2},2)$ has the most fixed points, and so is similarly first in CL order restricted to odd partitions.  When $\alpha = (n)$, $\beta \neq \alpha$, $i=\min_k(a_k \neq b_k) = \min_k(b_k \neq 0) > n$, so $a_i =0 < 1 \leq b_i$, so the $n$-cycle is last in CL order.  The $n$-cycle can be made with $n-1$ transpositions, and so has the same parity as $n-1$.  $(\frac{n}{2}^2)$ or $(\frac{n-1}{2},\frac{n+1}{2})$ similarly has the second to largest smallest cycle and will be the smallest element in CL order within its parity.  

-CL order is the reverse of CL order, so its largest element is the smallest element of CL order and so on.

Lastly, $(-1)^{i+1}$CL can be thought of as rewarding odd cycles and punishing even cycles.  Again, fixed points are the first deciding point.  When $\alpha = (1^n)$, $\beta \neq \alpha$, $\min_k (a_k \neq b_k) = 1$, and $a_1 > b_1$.  As $1$ is odd, having more odd cycles is good, and the identity is larger than any other element in this order.  Similarly restricted to odd conjugacy classes, the transposition $(1^{n-2},2)$ has more fixed points than any other element, and irrelevant of the $2$ cycle present will be first among odd element in this order. 

When $n$ is even, $\alpha = (2^{\frac{n}{2}})$ is last in this order since for $\beta \neq \alpha$, $\min_k (a_k \neq b_k)$ is either $1$ if $\beta$ has a fixed point or $2$ if not.  If $i=1$, $a_1 = 0 < 1 \leq b_1$, and $1$ is odd so $\beta$ is larger.  If $i=2$, $a_2 = \frac{n}{2} > b_2$, but $2$ is even so having more $2$-cycles makes a conjugacy class less likely under this order.  The next smallest element when $n$ is even is $(2^{\frac{n-4}{2}},4)$ since it is also fixed point free and has the next most $2$ cycles.  One of these two is odd while the other is even, and so they are the respective smallest odd and even elements.  

When $n$ is odd, the smallest element is the fixed point free conjugacy class with the most $2$ cycles, $(2^{\frac{n-3}{2}},3)$, so again in comparison with any other conjugacy class $i$ is $1$ or $2$.  $(2^{\frac{n-5}{2}},5)$ has the next most $2$-cycles amoungst fixed point free conjugacy classes.  Exactly one of these will be odd, the other even.  

\end{proof}

\section{Transposition Walk}\label{transp}

For the first walk, consider building a permutation by at each step appending a randomly selected transposition.  The goal is to find which permutations are more or less likely than others.  The answer for the transposition walk, is that that the likelihood order after sufficient time is given by cycle lexicographic order. The key is finding, given a pair of permutations, the partition indexing the largest character ratio with non-zero character difference in the decomposition given by Proposition \ref{bigformula}.  For the transposition walk, these partitions are $[n-i,i] = \lambda^i$ and $(\lambda^i)'$, where $i$ is as in the definition of cycle lexicographic order.  For $t>4.14n^2 + o(n^2)$ all of the order will hold.  Finally, the methodology will be extended to the lazy version of the walk, and the place in the likelihood order of the stationary distribution will be located.

The order holding after sufficient time is the best that can be hoped for as for all $n \geq 8$, the order breaks for some $t < n$.  For example when $n$ is divisble by $4$, the conjugacy class of an $(n-1)$-cycle is more likely after sufficient time than the conjugacy class of $n/2$ $2$-cycles. Yet, as it requires $n-2$ transpositions to achieve the former and only $n/2$ to achieve the latter, for $n/2 \leq t < n-2$ the order breaks.  All known cases, from simulation, of the order breaking are of this form where the order rights itself when the eventually more likely element is first possible.  There is not a known case of the order breaking for $t \geq n$.

\subsection{Character Ratio Maximizing $i$-Cycle Detector}

From Proposition \ref{bigformula}, the formula for the difference in probability of two conjugacy classes $\alpha,\beta$ is:

\[ P^{*t}(\alpha) - P^{*t}(\beta) = \sum_{\lambda} \left( \chi_{\lambda}(\alpha) - \chi_{\lambda}(\beta) \right) d_\lambda \left( \frac{\chi_{\lambda}(\tau)}{d_\lambda} \right)^t \]\label{tformula}

After sufficient time, the sign of this expression will be determined by the partitions $\lambda$ with the largest magnitude of character ratio, $\left|\frac{\chi_{\lambda}(\tau)}{d_\lambda}\right|$, and non-zero character difference $\chi_{\lambda}(\alpha) - \chi_{\lambda}(\beta)$.  It happens that when $\min_k (a_k \neq b_k) =i$ this lead position is taken by $\lambda^i, (\lambda^i)'$.  The relative sizes of character ratios at a transposition are well understood \cite{Diaconis},\cite{Ingram}:

\begin{prop}\label{tf}
When $\lambda \trianglerighteq \rho$, $\frac{\chi_{\lambda}(\tau)}{d_\lambda} \geq \frac{\chi_{\rho}(\tau)}{d_\rho}$.  Recall, $\lambda \trianglerighteq \rho$ when boxes in the Ferrers diagram of $\rho$ can be moved up and to the right to get the Ferrers diagram of $\lambda$.  The difference in character ratios is the distance the boxes travel divided by ${n \choose 2}$.
\end{prop} 

Recall that an $i$ cycle detector must have $h_{2,1}, h_{1,2} \geq i$. The goal is to find the $i$-cycle detector with largest positive character ratio.  For this walk, its conjugate will have the largest in magnitude character ratio amoung the $i$-cycle detectors with negative character ratios at a transposition.  

\begin{prop}
$\lambda^i = [n-i,i]$ is larger in majorization order than any other $i$-cycle detector, and so it and its conjugate are the $i$-cycle detectors with largest character ratios.
\end{prop}   
\begin{proof}
In order for $h_{2,1}$ to be $\geq i$, at least $i$ blocks must exist in the subhook.  The positioning of these most up and to the right is as in $[n-i,i]$. Moreover, for any partition with at least $i$ blocks in its subhook, these $i$ blocks can be moved up and to the right to be in the second row and all other blocks below the first row moved to the first row yielding $[n-i,i]$.  If any additional block is moved up and to the right from $[n-i,i]$, it must be moved out of the subhook making the partition no longer an $i$-cycle detector.  To be incomparible to $[n-i,i]$ in majorization order, a partition must have both more blocks in the first row and some blocks in the third row.  The condition of having more than $n-i$ blocks in the first row precludes it from being an $i$-cycle detector.  Thus, all $i$-cycle detectors are compariable to and below $[n-i,i]$ in majorization order.
\end{proof}

\subsection{Cycle Lexicographical Order}

The sign of \ref{tformula} after sufficient time is determined by the signs of the terms for $\lambda^i$ and $(\lambda^i)'$.  The $\lambda^i$ term will end up controlling the sign, while its conjugate contributes to the condition that only even partitions can be reached at even times and odd transpositions at odd times.

\begin{prop}
Let $\alpha,\beta$ be two conjugacy classes and $\alpha \geq_{CL} \beta$ with $min_k(a_k \neq b_k) = i$, then \[\chi_{\lambda^i}(\alpha) - \chi_{\lambda^i}(\beta) = a_i - b_i \]
\end{prop}
\begin{proof}
Consider the Murnaghan-Nakayama construction of these characters.  If the first cycle inserted is inserted starting in the second row, no $i$ or larger cycle can insert non-trivially.  So all $i$ and larger cycles in both $\alpha$ and $\beta$ would insert as one long cycle without changing the sign of the insertion.  All nontrivial insertions that follow are mimiced between $\alpha$ and $\beta$ since they have the same small cycles.  If the first cycle is inserted instead into the first row, the space remains for a single $i$ cycle to insert non-trivially into the second row, there are respecively $a_i$ and $b_i$ ways to do this.  All further cycles must insert trivially, and the insertion has a positive sign.  Again, if smaller cycles are inserted into the subhook, the insertions for $\alpha$ mimic exactly those for $\beta$ and cancel. Alternatively, one can compute the characteristic polynomial has a single monomial containing $x_i$, namely $x_i$.
\end{proof}

And quickly,

\begin{prop}
\[d_{\lambda^i} = { n \choose i}\frac{n-2i +1}{n-i+1} \] 
\end{prop}
\begin{proof}
By the hook length formula ~\cite{Stanley1}, \[ d_{\lambda^i} =  \frac{n!}{i!(n-i+1) \cdots (n-2i +2)(n-2i) \cdots 1} = \frac{n!}{i! (n-i)!}\frac{n-2i+1}{n-i+1} \] 
\end{proof}

\begin{prop}
\[\frac{\chi_{\lambda^i}(\tau)}{d_{\lambda^i}} = 1 - \frac{i(n-i+1)}{{n \choose 2}} >0 \] 
\end{prop}
\begin{proof}
Using the formula for the difference in character ratios in Proposition \ref{tf}, the difference between the character ratios of $\lambda^0$ (with character ratio $1$) and $\lambda^i$ is the distance in the Ferrers diagram the squares travel up and to the left from $\lambda^i$ to become $\lambda^0$, all divided by n choose $2$.  In this case, the $i$ blocks need to travel from the second row to the first row, for an average distance traveled of $(n-i+1)$.    Hence, \[\frac{\chi_{\lambda^i}(\tau)}{d_{\lambda^i}} = \frac{\chi_{\lambda^0}(\tau)}{d_{\lambda^0}} - \frac{i(n-i+1)}{{n \choose 2}} = 1 - \frac{i(n-i+1)}{{n \choose 2}} \]
\end{proof}

\begin{prop}
The sign of the sum of the $\lambda^i$ and $(\lambda^i)'$ terms in \ref{tformula} is positive when $t$ is the same sign as $\alpha,\beta$ and $\alpha \geq_{CL} \beta$ with $min_k(a_k \neq b_k) = i$. Thus, after sufficient time, CL is the likelihood order.
\end{prop}  
\begin{proof}
\begin{align*} &\left( \chi_{\lambda^i}(\alpha) - \chi_{\lambda^i}(\beta) \right) d_{\lambda^i} \left( \frac{\chi_{\lambda^i}(\tau)}{d_{\lambda^i}} \right)^t + \left( \chi_{(\lambda^i)'}(\alpha) - \chi_{(\lambda^i)'}(\beta) \right) d_{(\lambda^i)'} \left( \frac{\chi_{(\lambda^i)'}(\tau)}{d_{(\lambda^i)'}} \right)^t \\
&= (a_i - b_i) d_{\lambda^i} \left( \frac{\chi_{\lambda^i}(\tau)}{d_{\lambda^i}} \right)^t + \text{sgn}(\alpha)(a_i - b_i)d_{\lambda^i}\left( -\frac{\chi_{\lambda^i}(\tau)}{d_{\lambda^i}} \right)^t \\
&= (a_i-b_i)(1 + \text{sgn}(\alpha)(-1)^t)d_{\lambda^i} \left( \frac{\chi_{\lambda^i}(\tau)}{d_{\lambda^i}} \right)^t 
\end{align*}

Which is positive when $a_i > b_i$ and $\alpha, \beta$ are the same partity as $t$.
\end{proof}

\begin{prop}
After sufficient time the following hold.  The identity is the most likely element at even times, a transposition the most likely at odd times.  When $n$ is odd, an $n$-cycle is the least likely element at even times and an element of $(\frac{n-1}{2},\frac{n+1}{2})$ at odd times.  When $n$ is even, at even times an element of $(\frac{n}{2}^2)$ is least likely, while at odd times an $n$-cycle is least likely. 
\end{prop}
\begin{proof}
These are the largest and smallest elements of CL order under the required parity constrainsts, as in Proposition \ref{CL orders}.
\end{proof}

\subsection{Bound for Sufficient Time}

The general methodology will be to show that the $\lambda^i$ term is greater than the sum of the absolute value of all other terms for partitions $\lambda$ with $i \leq h_{2,1} \leq h_{1,2}$.  The symmetry between conjugates handles the remaining partitions.  The $\lambda^i$ term will be shown to be at least a scaling factor, $c_j$, times greater than all the $\lambda$ terms with $\lambda_1 = n-j$.  The expression to be shown is then is:

\[ c_j \left| \frac{\chi_{\lambda^i}(\alpha) - \chi_{\lambda^i}(\beta)}{{\max_{\lambda: \lambda_1=n-j} \chi_{\lambda}(\alpha) - \chi_{\lambda}(\beta)}} \right| \frac{d_{\lambda^i}}{\sum_{\lambda: \lambda_1=n-j} d_{\lambda}} \left| \frac{\frac{\chi_{\lambda^i}(\tau)}{d_{\lambda^i}}}{\max_{\lambda: \lambda_1=n-j} \frac{\chi_\lambda(\tau)}{d_\lambda}} \right|^t >1 \]

Next, at length, each of the above will be bounded.  First, the case for $i \leq \frac{n}{3}$ will be treated seperately than $i > \frac{n}{3}$.  Within $i \leq \frac{n}{3}$, the partitions $[n-i,i-k,1,...,1] = \lambda^{i,k}$ will be handled as a special case, as their character ratios are the closest to $\lambda^i$'s.  The remaining partitions will be handled based on the length of their first row.  For $i > \frac{n}{3}$ the closest character ratio to $\lambda^i$ is that of $\lambda^{i+1}$, and a different approach will handle this.  

These the representations with character ratios closest to that of $\lambda^i$ determine the upper bound. This gives the heuristic of $t \geq \max_{\lambda} \left(\frac{\chi_{\lambda^i}(\tau)}{d_{\lambda^i}}- \frac{\chi_\lambda(\tau)}{d_\lambda}\right)^{-1} \log(d_{\lambda}/d_{\lambda^i})$.

This is analogous to the heuristic estimate $\hat{\tau_n}$ of the mixing time of a random walk with sufficient symmetry on a group based on its spectral gap, $1- \lambda_2$ (where $1=\lambda_1,\lambda_2,...$ are the eigenvalues), and multiplicity of the second eigenvalue, $\delta_n$, ~\cite{AldousDiaconis}.  \[\hat{\tau_n} = \tau_e(n)\log(\delta_n)  \text{ where } \tau_e(n) = -1/\log(\lambda_2) \]  In these cases, the closest character will have a larger dimension than $\lambda^i$ by a factor between a $O(i)$ and $n$.  While the minimal character ratio difference, playing the part of the spectral gap, varies inversely at the same time between $O(\frac{O(i)}{n^2})$ and $O(\frac{1}{n})$  This analog of the heuristic describes all the bounds achieved for all the walks in this paper. 

\begin{prop}\label{ik}
The $\lambda^{i,k}$ term is:
\begin{align*} &\left( \chi_{\lambda^{i,k}}(\alpha) - \chi_{\lambda^{i,k}}(\beta) \right) d_{\lambda^{i,k}} \left( \frac{\chi_{\lambda^{i,k}}(\tau)}{d_{\lambda^{i,k}}} \right)^t \\& = (-1)^{k}(a_i - b_i){n \choose i}{i-1 \choose k} \frac{n-2i+k+1}{n-i+k+1} \left( 1 - \frac{i(n-i+k+1)}{{n \choose 2}} \right)^t \end{align*}
\end{prop}
\begin{proof}
From the Ferrers diagram of $\lambda^i$, to get to the Ferrers diagram of $\lambda^{i,k}$ $k$ boxes must be moved each a distance of $i$.  This gives the character ratio.  The hook length formula here gives \begin{align*}d_{\lambda^{i,k}} &= \frac{n!}{(n-i+k+1)(n-i) \cdots (n-2i + k+2)(n-2i +k) \cdots (1)(i)(i-k-1)!k!}\\ &= {n \choose i}{i-1 \choose k} \frac{n-2i+k+1}{n-i+k+1}\end{align*} And finally for the character difference, analagously to the Murnaghan-Nakayama insertions for $\lambda^i$, $i$-cycles can only insert non-trivially into the subhook exactly one way, and here that insertion will have sign $(-1)^k$. 
\end{proof}

\begin{lem}
Then the sum of all the $\lambda^{i,k}$ terms is \[ (a_i - b_i){n \choose i} \sum_{k=1}^{i-1}  (-1)^{k}{i-1 \choose k}\frac{n-2i+k+1}{n-i+k+1}\left( 1 - \frac{i(n-i+k+1)}{{n \choose 2}} \right)^t \].  This is less than half of the $\lambda^i$ term for $t \geq \frac{n^2(\log(i-1)+1)}{i}$. 
\end{lem}
\begin{proof}
For $t \geq \frac{n^2(\log(i-1)+1)}{i}$ consequtive terms will be shown to fall by half in magnitude.  Taken with the alternating signs this gives the desired result.  The $i=1$ case is trivial.  Consider $2 \leq i \leq n/3$:

\begin{align*} & \frac{(a_i-b_i){n \choose i}}{(a_i - b_i){n \choose i}}\frac{{i-1 \choose k}\frac{n-2i+k+1}{n-i+k+1}\left( 1 - \frac{i(n-i+k+1)}{{n \choose 2}} \right)^t}{{i-1 \choose k+1}\frac{n-2i+k+2}{n-i+k+2}\left( 1 - \frac{i(n-i+k+2)}{{n \choose 2}} \right)^t} \\
= & \frac{k}{i-k-1}\left(1 - \frac{1}{n-i+k+1}\right)\left(1+\frac{1}{n-2i+k+2}\right)\left( 1 + \frac{i}{{n \choose 2}(1 - \frac{i(n-i+k+2)}{{n \choose 2}})} \right)^t \\
\geq & \frac{1}{i-1}\left(1 - \frac{1}{n-i+1}\right)e^{t\log(1 + \frac{i}{{n \choose 2}-i(n-i+1)})} \\
\geq & \frac{3}{4(i-1)}e^{t\frac{i}{{n \choose 2}-i(n-i+1)}(1 - \frac{i}{2({n \choose 2}-i(n-i+1))})} \\
\geq & 2e^{t\frac{i}{\log(i-1)+ \log(3/2)}\frac{2}{{n \choose 2}-i(n-i+1)} } \\
\geq & 2
\end{align*}
\end{proof}

This is a central obstruction to CL order holding at the mixing time $\frac{1}{2}n\log(n)$.  When $i=1$ this issue does not exist since it is the sole term with $\lambda_1=n-1$. However, for all other $i$ simulation suggests that although the sum of the $\lambda^i$ and $\lambda^{i,k}$ terms may be positive by $O(n\log(n))$ time, it will not be a a constant fraction of the $\lambda^i$ term until $O(\frac{n^2\log(i)}{i})$. And by the time $i$ is $O(n)$, another impediment to $O(n\log(n))$ order holding will arise.

\begin{prop}
For $\lambda = [n-j,...]$ with $i < j \leq \frac{n}{2} $, \[ \frac{\chi_{\lambda^i}(\tau)}{d_{\lambda^i}} - \frac{\chi_{\lambda}}{d_{\lambda}} \geq \frac{\chi_{\lambda^i}(\tau)}{d_{\lambda^i}} - \frac{\chi_{\lambda^j}(\tau)}{d_{\lambda^j}} = \frac{(j-i)(n-(i+j) +1)}{{n \choose 2}} \geq \frac{(j-i)\frac{n}{6}}{{n \choose 2}}\] When $i \leq \frac{n}{3}$ this difference is at least $\frac{j-i}{3n}$.
\end{prop}
\begin{proof}
$\lambda^j$ has the largest character ratio at the transposition of all partitions with $j$ elements below the first row by majorization order.  Then the difference of character ratios of $\lambda^i$ and $\lambda^j$ is \[\left(1 - \frac{i(n-i+1)}{{n \choose 2}} \right) - \left(1 - \frac{j(n-j+1)}{{n \choose 2}}\right) = \frac{(j-i)(n-(i+j) +1)}{{n \choose 2}} \]
\end{proof}

\begin{prop}\label{jbig}
For $\lambda =[n-j,...]$ with $j > \frac{n}{2}$, \[ \frac{\chi_{\lambda^i}(\tau)}{d_{\lambda^i}} - \frac{\chi_{\lambda}}{d_{\lambda}} \geq \frac{j-i}{3n} \]
\end{prop}
\begin{proof}
The largest $\lambda = [n-j,...]$ in majorization order is the one whose Ferrers diagram has boxes as far right and up as possible, so the one with $\lfloor \frac{n}{n-j} \rfloor$ rows of $n-j$ squares and a final row of $n - (n-j) \lfloor \frac{n}{n-j} \rfloor$ blocks.  To find its character ratio, count how far these blocks have to travel to all be in the first row.  The first $n-j$ stay in place.  Then next $n-j$ each travel a distance of $n-j+1$ from the second row.  The third group of $n-j$ travel $2(n-j+1)$ from the third row, and so on, ending with the spare $n - (n-j) \lfloor \frac{n}{n-j} \rfloor$ blocks each traveling $\lfloor \frac{n}{n-j}(n-j+1) \rfloor$.  This gives a total distance traveled of 
\begin{align*}& {n \choose 2}\left(1 - \frac{\chi_{\lambda}(\tau)}{d_{\lambda}}\right) \\
&\geq 0(n-j)(n-j+1) + 1((n-j)(n-j+1))+... + \lfloor \frac{n}{n-j} -1 \rfloor(n-j)(n-j+1) + \\
&\lfloor \frac{n}{n-j} \rfloor(n-j+1) \left(n - (n-j) \lfloor \frac{n}{n-j} \rfloor \right) \\
=& \frac{1}{2}\left(\lfloor \frac{n}{n-j} \rfloor -1\right)\left(\lfloor \frac{n}{n-j} \rfloor\right)(n-j+1)(n-j) +\left(\lfloor \frac{n}{n-j} \rfloor\right)(n-j+1) \left(n - (n-j) \lfloor \frac{n}{n-j} \rfloor \right) \\
=& \left(\lfloor \frac{n}{n-j} \rfloor\right)(n-j+1)\left(\frac{1}{2}\left(\lfloor \frac{n}{n-j} \rfloor -1\right)(n-j) +   \left(n - (n-j) \lfloor \frac{n}{n-j} \rfloor \right)\right) \\
=& \left(\lfloor \frac{n}{n-j} \rfloor\right)(n-j+1)\left(n - \frac{1}{2}\left(\lfloor \frac{n}{n-j} \rfloor +1\right)(n-j)\right) \\
=& \frac{1}{2}\left(\lfloor \frac{n}{n-j} \rfloor\right)(n-j+1)\left(n+j - \lfloor \frac{n}{n-j} \rfloor(n-j)\right)\\
\geq& \frac{1}{2}j(n+2)  \end{align*}

Then,

\begin{align*} \frac{\chi_{\lambda^i}(\tau)}{d_{\lambda^i}} - \frac{\chi_{\lambda}}{d_{\lambda}} &\geq \frac{\frac{1}{2}j(n+2) - i(n-i+1)}{{n \choose 2}} \\
&= \frac{\frac{1}{2}(j-i)(n+2) - \frac{n}{2} - i(\frac{n}{2} - i)}{{n \choose 2}} \\
&\geq \frac{(j-i)(n/2 - i +1)}{{n \choose 2}} \\
&\geq \frac{j-i}{3n} \end{align*}

\end{proof}

\begin{prop}
Taking the sum over $\lambda = [n-j,...]$ with $j>i$ and $h_{2,1} \geq i$, \[\sum_{\lambda} d_{\lambda} \leq {{n \choose j}}{{j \choose a,b}}e_2(j-i)\] where if $j \geq \frac{3}{2}i$, $a=b=j/3$ otherwise $a=b=i/2$.
\end{prop}
\begin{proof}
The key here is to use that $d_\lambda$ counts the number of standard Young tableau of $\lambda$. To upperbound the sum, $\lambda$ will be seperated into three pieces. The first piece is $\lambda_1=n-j$. The second piece is $h_{2,1}$, of which it is restricted to $i \leq h_{2,1} \leq j$ and is made of $\lambda_2$ and $\lambda_1'-2$. This leaves a partition consisting of all blocks not in $\lambda_1,\lambda_2,\lambda_1'$, which contains at most $j-i$ blocks. Consider all possible ways of placing $1,2,...,n$ into these parts so that each row/column inside part is increasing. Not all of these will give valid tableau but all possible tableau will be present. First choose $n-j$ of $n$ to place in $\lambda_1$. Then $\lambda_2$ and $\lambda_1'$ of the remaining $j$. Let $a,b$ be the values of $\lambda_2,\lambda_1'-2$ that maximize this over all $\lambda$. Finally this leaves a partition of size at most $j-i$. The sum of the dimensions of all partitions of $j-i$ is a combinatorial object called a telephone number and denoted $e_2(j-i)$ (the name follows from a bijection with involutions of $S_n$). From \cite{Knuth3}, $e_2(j-i) \leq \left(\frac{j-i}{e}\right)^{(j-i)/2} e^{\sqrt{j-i}}$. This gives 
\[\sum_{\lambda} d_{\lambda} \leq {{n \choose j}}{{j \choose a,b}}e_2(j-i)\]
This leaves finding values of $a,b$. ${j \choose a,b,j-a-b}$ is maximized when the parts are as equal as possible in size, so clearly $a=b$. Also, $i \leq a+b \leq j$. So if $j \geq \frac{3}{2}i$, all parts can be equal letting $a=b=j/3$. Otherwise, it is maximized when $a=b=i/2$, letting the remaining part be as large as possible.  
\end{proof}

\begin{prop}
For $\lambda = [n-j,...]$ with $j>i$ and a $i$-cycle detector and $\alpha,\beta$ conjugacy classes of the same parity with $min_k(a_k \neq b_k) = i$,
\[ \chi_{\lambda}(\alpha) - \chi_{\lambda}(\beta) \leq (n-2i)_{j-i}(j-i+1)^{2}\frac{n-i}{i}\]
\end{prop}
\begin{proof}
This bound is found by bounding the valid insertions from Murnaghan-Nayama containing non-trivial $i$-cycle insertions independent of sign. Each insertion is determined by how the cycles that insert below the first row of $\lambda$ insert.  The first insertion must be trivial and at least $i$ in length, as both cycles must contain an $i$ or larger cycle.  There are at most $j-i+1$ places in the first row to start this insertion leaving room for an $i$-cycle. Any insertions with no $i$ or larger cycles inserted non-trivially, the first of which is below the first row, will cancel in the difference, so the sum in Murnaghan-Nakayama is restricted to insertions where at least one $i$ or larger cycle is inserted non-trivially.  This means that one of $\alpha,\beta$ has at least two cycles of size $\geq i$. Then since $\sum_{j<i} j(a_j) = \sum_{j<i} j(b_j)$ both are at most $n-2i$. Following this first non-trivial $i$ or larger cycle at most $j-i$ cycles can then additionally insert below the first row and there are at most $n-2i$ of these. There are at most $(j-i)!$ ways of arranging these. And there are at most $\frac{n}{i} - 1$ choices for the first non-trivial $i$ or larger cycle and $j-i+1$ rows in which to insert it.    
\end{proof}

\begin{prop}
For $\lambda = [n-j,...]$ with $i < j$, $\alpha, \beta$ conjugacy classes of the same parity with $min_k(a_k \neq b_k) = i$, $i \leq \frac{n}{3}$, 

\[ \left| \frac{\frac{\chi_{\lambda^i}(\tau)}{d_{\lambda^i}}}{\frac{\chi_\lambda(\tau)}{d_\lambda}} \right|^t \geq e^{t\frac{(j-i)}{6n}} \]
\end{prop}
\begin{proof}
\begin{align*} \left| \frac{\frac{\chi_{\lambda^i}(\tau)}{d_{\lambda^i}}}{\frac{\chi_\lambda(\tau)}{d_\lambda}} \right|^t 
&\geq \left( 1 + \frac{\frac{\chi_{\lambda^i}(\tau)}{d_{\lambda^i}} - \frac{\chi_{\lambda}(\tau)}{d_\lambda}}{\frac{\chi_\lambda(\tau)}{d_\lambda}} \right)^t \\ 
& \geq \left( 1 + \frac{\chi_{\lambda^i}(\tau)}{d_{\lambda^i}} - \frac{\chi_{\lambda}(\tau)}{d_\lambda}\right)^t \\
& \geq e^{t \log{\left(1 + \frac{(j-i)}{3n}\right)}} \\
&\geq e^{t \frac{(j-i)}{3n}\left(1 - \frac{1}{2}\frac{(j-i)}{3n} \right)} \\ 
& \geq e^{t \frac{(j-i)}{3n}(1 - \frac{1}{2})} \\
& \geq e^{t \frac{(j-i)}{6n}}
\end{align*}
\end{proof}

And the last part for $i \leq \frac{n}{3}$ is the scaling constant $c_j$.  Recall that half of the value of the $\lambda^i$ term was used to counter the $\lambda^{i,k}$ terms so $c_i = \frac{1}{2}$. Then for each additional $j >i$, let $c_j = 2^{-(j-i+1)}$.

\begin{lem}
For $n \geq 9$ and $t \geq 2\log(2)n^2 + 24n\log(n) +36n$ with $i \leq \frac{n}{3}$, $\alpha,\beta$ of conjugacy classes of the same parity with $\min_{a_k \neq b_k} = i$, $\lambda = [n-j,...]$, $j>i$, the $\lambda^i$ term is larger than the sum of all terms with $\lambda_1 = n-j$ by a factor of $\frac{1}{c_j}$.
\end{lem}
\begin{proof}

First the case where $j \geq \frac{3}{2}i$. Note this means that $(j-i) \geq j/3 \geq i/2$. Take $t \geq 24n\log(n-i+1) + 48n$. For $n \geq 9$ this gives $t \geq 2\log(2)n^2+24n\log(n) +36n$.
\begin{align*} & c_j \left| \frac{\chi_{\lambda^i}(\alpha) - \chi_{\lambda^i}(\beta)}{{\max_{\lambda} \chi_{\lambda}(\alpha) - \chi_{\lambda}(\beta)}} \right| \frac{d_{\lambda^i}}{\sum_{\lambda} d_{\lambda}} \left| \frac{\frac{\chi_{\lambda^i}(\tau)}{d_{\lambda^i}}}{\max_{\lambda} \frac{\chi_\lambda(\tau)}{d_\lambda}} \right|^t  \\
&\geq  2^{-(j-i+1)} (n-2i)_{j-i}^{-1}\frac{1}{(j-i+1)^2}\frac{i}{n-i} {{n \choose i}}\frac{n-2i+1}{n-i+1}{{n \choose j}}^{-1}3^{-j}e_2(j-i)^{-1}\left(1 + \frac{j-i}{3n}\right)^t \\
&\geq e^{-\left((j-i+1)\log(2) + (j-i)\log({n-2i}) + 2\log(j-i+1) \right)} \\
& e^{(-\left((\log{\frac{(n-i)(n-i+1)}{i(n-2i+1)}} + (j-i)\log\frac{e(n-i)}{j} + j\log(3) + \frac{j-i}{2}\log\frac{j-i}{e} + \sqrt{j-i} \right) + t\log(1 + \frac{j-i}{3n})} \\
& \geq e^{-\left( (j-i)\left(ln(6) + 3 + \frac{1}{\sqrt{j-i}} + \log{\frac{(n-2i)(n-i)\sqrt{j-i}}{j}} \right) + \log\frac{(j-i+1)^2(n-i)(n-i+1)}{i(n-2i+1)} + i\log3\right) \frac{t(j-i)}{6n}} \\
& \geq e^{-\left( (j-i)\left(ln(54) + 4 + \log{\frac{(n-2i)(n-i)\sqrt{j-i}}{j}}\right) + 	\log\frac{(j-i+1)^2(n-i)(n-i+1)}{i(n-2i+1)}\right) + \frac{t(j-i)}{6n}}	\\
& \geq e^{-\left( (j-i)\left(8 + 2\log{n-i}\right) + 2\log{(j-i)(n-i+1)}\right) + \frac{t(j-i)}{6n}} \\
& \geq e^{(j-i) \left(-(8 + (2+ \log(2))\log(n-i+1)) + \frac{t}{6n}\right)} \\
& > 1
\end{align*}

Now the case where $j < \frac{3}{2}i$. So $j-i \leq i/2$. Take $t \geq 6n(i\log2 + 4\log(n) + 6)$. When at worst $i = n/3$ this gives $t \geq 2\log(2)n^2 + 24n\log(n) +36n$.
\begin{align*} & c_j \left| \frac{\chi_{\lambda^i}(\alpha) - \chi_{\lambda^i}(\beta)}{{\max_{\lambda} \chi_{\lambda}(\alpha) - \chi_{\lambda}(\beta)}} \right| \frac{d_{\lambda^i}}{\sum_{\lambda} d_{\lambda}} \left| \frac{\frac{\chi_{\lambda^i}(\tau)}{d_{\lambda^i}}}{\max_{\lambda} \frac{\chi_\lambda(\tau)}{d_\lambda}} \right|^t  \\
&\geq  2^{-(j-i+1)} (n-2i)_{j-i}(j-i+1)^{-2} \frac{i}{n-i} {{n \choose i}}\frac{n-2i+1}{n-i+1}{{n \choose j}}^{-1} \\
&{{j \choose i/2,i/2,j-i}}^{-1}e_2(j-i)^{-1} \left(1 + \frac{j-i}{3n}\right)^t \\
&\geq e^{-\left((j-i+1)\log(2) + (j-i)\log({n-2i}) + 2\log(j-i+1) \right)} \\
& e^{-\left( \log{\frac{(n-i)(n-i+1)}{i(n-2i+1)}} + (j-i)\log\frac{e(n-i)}{j} + i\log\frac{2j}{i} + (j-i)\log\frac{j}{j-i} + \frac{j-i}{2}\log\frac{j-i}{e} + \sqrt{j-i} \right) + t\log(1 + \frac{j-i}{3n})}\\
& \geq e^{-\left( (j-i)\left(ln(2) + 3 + \frac{1}{\sqrt{j-i}} + \log{\frac{(n-2i)(n-i)\sqrt{j-i}j}{j(j-i)}} \right) + \log\frac{(j-i+1)^2(n-i)(n-i+1)}{i(n-2i+1)} + i\log\frac{2j}{i}\right) + \frac{t(j-i)}{6n}} \\
& \geq e^{-\left( (j-i)\left(ln(2) + 4 + \log{(n-2i)(n-i)}\right)+ i\log\frac{2j}{i} +\log\frac{(j-i+1)^2(n-i)(n-i+1)}{i(n-2i+1)}\right) + \frac{t(j-i)}{6n}}	\\
& \geq e^{-\left( (j-i)\left(6 + (2 + \log(2))\log{n-i}\right)+ i\log2\right) + \frac{t(j-i)}{6n}} \\
& > 1
\end{align*}

\end{proof}

For $i > n/3$ the partition with the closest character ratio at a transposition to $\lambda^i$ is $\lambda^{i+1}$.  In the case that $i = n/2-1$, this once again attains the worst possible character ratio difference of $\frac{2}{{n \choose 2}}$.  Using the same formulas from before:

\begin{prop}
For $\frac{n}{3} < i < j \leq \frac{n}{2}$,
\[ \frac{\chi_{\lambda^i}(\tau)}{d_{\lambda^i}} - \frac{\chi_{\lambda^j}(\tau)}{d_{\lambda^j}} = \frac{(j-i)(n-(i+j)+1)}{{n \choose 2}} \geq \frac{(j-i)(n/2-i+1)}{{n \choose 2}} \]
\end{prop}

\begin{prop}
For $\frac{n}{3} < i$, $0<k<i$,
\[ \frac{\chi_{\lambda^i}(\tau)}{d_{\lambda^i}} - \frac{\chi_{\lambda^{i,k}}(\tau)}{d_{\lambda^{i,k}}} \geq \frac{ik}{{n \choose 2}} \]
\end{prop}

\begin{prop}
For $\frac{n}{3} \leq i < j \leq \frac{n}{2}$, for the $\lambda$ with $\chi_{\lambda}(\alpha) - \chi_{\lambda}(\beta) \neq 0$, $\lambda = [n-j,...] \neq \lambda^j$ when 
\[ \frac{\chi_{\lambda^i}(\tau)}{d_{\lambda^i}} - \frac{\chi_{\lambda}(\tau)}{d_{\lambda}} \geq  \frac{(j-i)(n/2-i+1) +j}{{n \choose 2}}\]
\end{prop}
\begin{proof}
This follows from moving blocks down from $\lambda^i$ to get $\lambda^j$ and then at least one additional block down into the third row. This extra block must move distance at least $j$.
\end{proof}

\begin{prop}
For $\frac{n}{3} \leq i \leq \frac{n}{2} < j $, for the $\lambda$ with $\chi_{\lambda}(\alpha) - \chi_{\lambda}(\beta) \neq 0$, $\lambda = [n-j,...]$ when 
\[ \frac{\chi_{\lambda^i}(\tau)}{d_{\lambda^i}} - \frac{\chi_{\lambda}(\tau)}{d_{\lambda}} \geq  \frac{(j-i)(n/2-i+1) +i(j-n/2)}{{n \choose 2}}\]
\end{prop}
\begin{proof}
This follows from the distance calculated from moving from $\lambda^i$ to the $\lambda$ with $\lambda_1=n-j$ that is first in majorization order in \ref{jbig}. This gives the difference is at least:
\[ \frac{ \frac{1}{2}j(n+2) -i(n-i+1)}{{n \choose 2}} = \frac{(j-i)(\frac{n}{2} -i+1) + i(j - \frac{n}{2})}{{n \choose 2}} \]
\end{proof}

The bounds for $\frac{d_{\lambda^i}}{d_{\lambda}}$ and $\log(\chi_{\lambda}(\alpha) - \chi_{\lambda}(\beta)) $ still hold.  Further, \[\frac{d_{\lambda^i}}{d_{\lambda^j}} \geq e^{-(j-i)\log(\frac{n-i}{j+1})} \geq e^{-(j-i)\log(2)}\]  And $\log(\chi_{\lambda^j}(\alpha) - \chi_{\lambda^j}(\beta)) \leq (n-2i)_{j-i}2\frac{n-i}{i}$ since below the first row there are no options as to placement and only the trivial insertion has $2$ places. Now the cases are for each $k$ of $\lambda^{i,k}$ taking $c_{i,k} = 2^{-k}$, one each for $\lambda^j$ for $i<j\leq n/2$ with $c_{j'} = 2^{-(2+(j-i))}$, and finally the remaining $\lambda$ with $\lambda_1 = n-j$ with $c_j = 2^{-(2+j-i)}$. Note, only $j \leq n- i/2$ need to be considered since an $i$ cycle needs to be trivially insertable. Putting it all together,

\begin{lem} 
For $t \geq 4.14n^2 + 6n\log(n) + 6\log(2)n$, $n/3 < i \leq j$, the $\lambda^i$ term is greater by a factor of $\frac{1}{c_j}$ than the $\lambda$ term with $\lambda = [n-j,...]$.
\end{lem}  
\begin{proof}
Three cases of $ \lambda^{i,k}$, $\lambda^j$, and $\lambda$.  

For $\lambda^{i,k}$, $t \geq \frac{\log(2i)}{i} n(n-1)$:
\begin{align*} c_i \left| \frac{\chi_{\lambda^i}(\alpha) - \chi_{\lambda^i}(\beta)}{\chi_{\lambda^{i,k}}(\alpha) - \chi_{\lambda^{i,k}}(\beta)} \right| \frac{d_{\lambda^i}}{d_{\lambda^{i,k}}} \left| \frac{\frac{\chi_{\lambda^i}(\tau)}{d_{\lambda^i}}}{\frac{\chi_{\lambda^{i,k}}(\tau)}{d_{\lambda^{i,k}}}}  \right|^t & \geq e^{-k\ln(2) - k\log(i) + t\frac{ik}{n(n-1)}} \\
& \geq e^{ - k\log(2i) + k\log(2i)} \\
& = 1 \end{align*}
So $t \geq 3n\log n$ is sufficient for any $i$.

For $\lambda^j$, $t \geq n^2\left(2\frac{\log(n-2i) + 2}{n-2i + 2} + 2\right) $
\begin{align*} c_{j'} \left| \frac{\chi_{\lambda^i}(\alpha) - \chi_{\lambda^i}(\beta)}{\chi_{\lambda^j}(\alpha) - \chi_{\lambda^j}(\beta)} \right| \frac{d_{\lambda^i}}{d_{\lambda^j}} \left| \frac{\frac{\chi_{\lambda^i}(\tau)}{d_{\lambda^i}}}{\frac{\chi_{\lambda^j}(\tau)}{d_{\lambda^j}}}  \right|^t & \geq e^{-(j-i+2)\log(2) - (j-i)\log(n-2i) - 2 - t\frac{(j-i)(n/2 - i +1)}{n(n-1)}} \\
& \geq e^{-(j-i)\log(2(n-2i)) - (2 + 2\log(2)) + (j-i)\log(2(n-2i)) + 2(j-i)(n/2 -i +1) } \\
& \geq 1 \end{align*}

For $\lambda = [n-j,...] \neq \lambda^j$, when $j \leq n/2$ for $t \geq \max{\left(n^2\left(2\log(2) + \frac{1}{2} + \frac{1}{e}\right),\log(3)n^2 + 9n\log(n)\right)}$:
\begin{align*} &c_j \left| \frac{\chi_{\lambda^i}(\alpha) - \chi_{\lambda^i}(\beta)}{\chi_{\lambda}(\alpha) - \chi_{\lambda}(\beta)} \right| \frac{d_{\lambda^i}}{d_{\lambda}} \left| \frac{\frac{\chi_{\lambda^i}(\tau)}{d_{\lambda^i}}}{\frac{\chi_\lambda(\tau)}{d_\lambda}}  \right|^t \\
& \geq 2^{-(j-i+1)}(n-2i)_{j-i}^{-1}(j-i+1)^2\frac{i}{n(n-i)}{n \choose i}{n \choose j}^{-1}3^{-j}e_2(j-i)^{-1} \\
& \left(1 + \frac{(j-i)(n/2-i+1) + j}{{n \choose 2}}\right)^t \\
& \geq e^{-\left((j-i+2)\log(2) + (j-i)\log(n-2i) + 2\log(j-i+1) + \log(n(n-i)/i) + (j-i)\log(\frac{n-i}{j}) + j\log(3) + \frac{j-i}{2}(\log(j-i)-1) + \sqrt{j-i} \right)}\\
& e^{t{(j-i)(n/2-i+1) +j}{{n(n-1)}}} \\
& \geq e^{-\left((j-i)\left(\log(2) + \log(n-2i) +  \log\frac{n-i}{j} + \frac{1}{2}\log(j-i) - \frac{1}{2} + 1\right) + 2\log(j-i+1) + 2\log(2) + \log(2n)+ j\log(3) \right)} \\
& e^{ t\frac{(j-i)(n/2-i+1)+j}{n(n-1)}}\\
& \geq e^{-\left((j-i)(n/2-i+1) \left(\ln(2) + \frac{\log(n-2i)}{2(n-2i+1)} + \log(2) + \frac{\log(n/2-i)}{n/2-i+1} + \frac{1}{2}\right) + j\left(\frac{2\log(j-i+1}{j}+ \frac{3\log2}{j} + \frac{\log(j)}{j} + \log(3)\right)\right)} \\
& e^{ t\frac{(j-i)(n/2-i+1)+j}{n(n-1)}} \\
& \geq e^{-\left((j-i)(n/2-i+1)\left(2\log(2) + \frac{1}{2} + \frac{1}{e}\right) + j\left(\frac{9\log(n)}{n} + \log(3)\right) \right) + t\frac{(j-i)(n/2-i+1)+j}{n(n-1)}} \\
& \geq 1 \end{align*}
This is at worst $t>2.26n^2 + 9n\log(n)$.

For $\lambda = [n-j,...] \neq \lambda^j$, when $j > n/2$, for \[t \geq \max{n^2\left(\log(24)+\frac{5+2e}{4e}\right), \log(3)n^2 +6n\log(n) + 6\log(2)n} \]
\begin{align*} &c_j \left| \frac{\chi_{\lambda^i}(\alpha) - \chi_{\lambda^i}(\beta)}{\chi_{\lambda}(\alpha) - \chi_{\lambda}(\beta)} \right| \frac{d_{\lambda^i}}{d_{\lambda}} \left| \frac{\frac{\chi_{\lambda^i}(\tau)}{d_{\lambda^i}}}{\frac{\chi_\lambda(\tau)}{d_\lambda}}  \right|^t \\
& \geq 2^{-(j-i+2)}(n-2i)_{j-i}^{-1}(j-i+1)^2\frac{i}{n(n-i)}3^{-j}e_2(j-i)^{-1}{n \choose i}{n \choose j}^{-1} \\
&\left(1 + \frac{(j-i)(n/2-i+1) + i(j-n/2)}{{n \choose 2}}\right)^t \\
& \geq e^{-\left((j-i+2)\log(2) - (j-i)\log(n-2i) + 2\log(j-i+1) + \log\frac{n(n-i)}{i} + j\log3 + 1/2(j-i)(\log(j-i)-1) + \sqrt{j-i} + (n-i-j)\log2 \right)} \\ 
& e^{ t\frac{(j-i)(n/2-i+1) + i(j-n/2)}{{n(n-1)}}} \\
& \geq e^{-(j-i)\left(\log(2)+\log(n-2i) + \frac{j-i+1}{j-i} + \log(3) + \frac{n-2i}{4}\log(j-i) -\frac{1}{2} + \frac{1}{\sqrt{j-i}}\right) - (n-2i)\log(2)- \frac{1}{2}(j-\frac{n}{2})\log(j-i) - i \log(3)- \log(8n)} \\
&e^{t\frac{(j-i)(n/2-i+1) + i(j-n/2)}{{n(n-1)}}} \\
& \geq e^{-(j-i)(n-2i+1)\left(\log(2) + \frac{n-2i}{n-2i+1} + \log(2) + \log(3) + \frac{1}{4e} - \frac{1}{2} + 1 + \log(2) \right) - i(j - \frac{n}{2})\left( \frac{j-i}{i} + \log(3) + \frac{\log(8n)}{i} \right)} \\
& e^{t\frac{(j-i)(n/2-i+1) + i(j-n/2)}{{n(n-1)}}} \\
& \geq e^{-(j-i)(n-2i+1)\left(ln(24) + \frac{5+2e}{4e}\right) - i(j-n/2)\left(\log(3) + \frac{6\log(n)}{n} + \frac{6\log(2)}{n}\right) + t\frac{(j-i)(n/2-i+1) + i(j-n/2)}{{n(n-1)}}}\\
& \geq 1 \end{align*}
\end{proof}

The lemmas for $i<n/3$ and $\lambda^{i,k}$, $i<n/3, j>i$, and $i > n/3$ together give that:

\begin{thm}
After taking \[t> \max\left(4.14n^2 + 6n\log(n) + 6\log(2)n, 9, 2\log(2)n^2 + 24n\log(n) +36n \right)\] steps, cycle lexicographical order holds between all conjugacy classes for the transposition walk. 
\end{thm}

\subsection{The Lazy Version}\label{lazy}

If the walk is modified to be lazy, all the above results hold with a slight modification to the bound for sufficient time.  If during a step the chain stays with probability $p$, and moves via a transposition with probability $1-p$, the probability function becomes as from proposition \ref{bigformula}:

\[P^{*t}(\alpha) = \sum_{\lambda} \chi_{\lambda}(\alpha) d_{\lambda} \left(p + (1-p) \frac{\chi_{\lambda}(\tau)}{d_\lambda} \right)^t \]

The lead $i$-cycle detector is now solely $\lambda^i$ instead of jointly $\lambda^i$ and $(\lambda^i)'$.  This removes the sign constraint, and CL order holds over all permutations after sufficient time.

This modifies only the portion of the bound for the exponential term, but adds back in the need to handle the conjugates previously excluded by symmetry. The new mixing time will be $\frac{1}{p(1-p)}$ times the previous bounds.  There are three cases to handle.  For the partitions with positive character ratios at the transposition, $\frac{1}{1-p}$ times the old time bound suffices 

\[ \left( p + (1-p) \frac{\chi_{\lambda^i}(\tau)}{d_{\lambda^i}}\right) - \left(p + (1-p) \frac{\chi_{\lambda}(\tau)}{d_\lambda} \right) = (1-p)\left( \frac{\chi_{\lambda^i}(\tau)}{d_{\lambda^i}} - \frac{\chi_{\lambda}(\tau)}{d_\lambda} \right) \]

For those with negative character ratios at the transposition, there are two cases, for whether the new character ratio is negative or positive. Either:

\begin{align*} \left( p + (1-p) \frac{\chi_{\lambda^i}(\tau)}{d_{\lambda^i}}\right) - \left|\left(p + (1-p) \frac{\chi_{\lambda'}(\tau)}{d_{\lambda'}} \right)\right| &= \left( p + (1-p) \frac{\chi_{\lambda^i}(\tau)}{d_{\lambda^i}}\right) - \left(p -(1-p) \frac{\chi_{\lambda'}(\tau)}{d_{\lambda'}} \right) \\
&= (1-p)\left(\frac{\chi_{\lambda^i}(\tau)}{d_{\lambda^i}} + \frac{\chi_{\lambda}(\tau)}{d_\lambda}\right) \end{align*}

\begin{align*} \left( p + (1-p) \frac{\chi_{\lambda^i}(\tau)}{d_{\lambda^i}}\right) - \left|\left(p + (1-p) \frac{\chi_{\lambda'}(\tau)}{d_{\lambda'}} \right)\right| &= \left( p + (1-p) \frac{\chi_{\lambda^i}(\tau)}{d_{\lambda^i}}\right) - \left((1-p) \frac{\chi_{\lambda'}(\tau)}{d_{\lambda'} -p} \right) \\
&= 2p + (1-p)\left(\frac{\chi_{\lambda^i}(\tau)}{d_{\lambda^i}} - \frac{\chi_{\lambda}(\tau)}{d_\lambda} \right) \end{align*}

Both work with the general bounds from before. Halfing $c_j$ to account for double the terms has a trivial impact on the time bounds.

\subsection{More or Less Likely than Stationary}\label{uniform}

This style of analysis can also be insightful into total variation distance as well as separation distance.  One of the equivalent definitions of total variation distance between the walk at time $t$ and stationary is \[||P^{*t} - \pi||_{TV} = \sum_{\sigma \in S_n, P^{*t}(\sigma) \geq \pi(\sigma)} [P^{*t}(\sigma) - \pi(\sigma)] \] So knowing the permutations that are more likely than stationary leads to being able to calculate the total variation distance for each of the variants on the transposition walk discussed: non-lazy at even times, non-lazy at odd times, and lazy.  The stationary distributions are respectively uniform over even permutations, odd permutations, and all permutations.  The stationary distribution is seen in the Fourier inversion formula in the $[n],[1^n]$ terms.  After sufficient time, the sign of the next non-zero term determines whether the Fourier inversion formula evaluated at a permutation will be more or less likely than uniform.  

The next non-zero term will be the first of $[n-1,1],[n-2,2],[n-2,1,1]$ with non-zero character at $\alpha$.  

\[\begin{array}{c|c|c|c|c|c|c} \lambda & \chi_{\lambda}(\alpha) & a_1>1 & a_1 =0 & a_1 = 1, a_2 >1 & a_1 = 1, a_2 = 0 & a_1 = 1, a_2 = 1 \\ \hline
n-1,1 & a_1 -1 & >0 & <0 & 0 & 0 & 0 \\
n-2,2 & a_2 + {a_1 \choose 2} -a_1 & & & >0 & <0 & 0 \\
n-2,1,1 & -a_2 + {a_1 \choose 2} - a_1 + 1 & &&&& <0 \end{array} \]

So after sufficient time, a permutation is more likely than uniform if it has at least two fixed points, or one fixed point and at least two $2$-cycles.  And a permutation is less likely than uniform if it has no fixed points or one fixed point and at most one $2$-cycle.  Now to find a time when this holds.  Unlike the whole order holding, most of the order will be of the same order as mixing but with the wrong constant.  

This will be similar to the above analysis but with only $\lambda^1,\lambda^2, \lambda^{2,1}$ dominating, and with Proposition \ref{fformula} not Proposition \ref{bigformula} $\chi_\lambda(\alpha)$ replacing $\chi_{\lambda}(\alpha) - \chi_{\lambda}(\beta)$).  Paralleling the method from before.  First, there is only one case of $\lambda^{i,k}$ to handle.  Namely comparing $\lambda^2$ to $\lambda^{2,1}$, where the former is primary and latter term non-zero in the case that $a_1 =1, a_2>1$.  This is the only case where the order can only be shown for $t> n^2$.  The rest of the order holds for $8n\log(n)$.  

\begin{prop}
For $t \geq n^2$ in the case that $a_1 = 1$, $a_2 >1$, \[\left|\chi_{\lambda^{2}}(\alpha)d_{\lambda^2} \left(\frac{\chi_{\lambda^2}(\tau)}{d_{\lambda^{2}}}\right)^t \right|> 2 \left|\chi_{\lambda^{2,1}}(\alpha)d_{\lambda^{2,1}} \left(\frac{\chi_{\lambda^{2,1}}(\tau)}{d_{\lambda^{2,1}}}\right)^t\right| \]
\end{prop}
\begin{proof}
This amounts to the equation:
\[|a_2 -1|\frac{n-3}{n-1}\left(1 - \frac{2(n-1)}{{n \choose 2}} \right)^{n^2} \geq 2|-a_2| \frac{n-2}{n}\left(1 - \frac{2(n)}{{n \choose 2}} \right)^{n^2}\]

In the worst case scenario of $a_2 =2$, the coeffcient on the right can be double that of the left.  This makes the extra time unavoidable by this method.

It suffices then to show that \[\frac{n(n-3)}{4(n-2)(n-1)}\left(1 + \frac{4}{n(n-1)} \right)^{n^2} \geq 1 \]

\[\frac{n(n-3)}{4(n-2)(n-1)}\left(1 + \frac{4}{n(n-1)} \right)^{n^2}  \geq e^{-\log(4) + \log(1 - \frac{2}{(n-2)(n-1)}) + n^2\frac{2}{n(n-1)}} \geq 1 \]

\end{proof}

For comparing $\rho = \lambda^1, \lambda^2, \lambda^{2,1}$ with all $\lambda = [n-j,...]$ with $3 \leq j \leq \frac{n}{2}$, bound the ratio of characters and ratio of dimension at $\alpha$ by the sum of the dimensions of the $\lambda$ with $\lambda=n-j$, $c_j = 2^{-j}$, and  $\frac{\chi_{\rho}(\tau)}{d_{\rho}} \geq  \frac{\chi_{\lambda{2,1}}(\tau)}{d_{\lambda^{2,1}}} = 1 - \frac{4}{n-1}$.  This means that \[\frac{\chi_{\lambda{2,1}}(\tau)}{d_{\lambda^{2,1}}} - \frac{\chi_{\lambda}(\tau)}{d_{\lambda}} \geq \frac{(j-2)(n-j-1) -2}{{n \choose 2}} \]

\begin{prop}
When $a_1 \neq 1$, the $\lambda^1$ term is larger than the $\lambda^{2}$ and $\lambda^{2,1}$ terms by a factor of $4$ by $t \geq 4n\log(n)$ for $n \geq 4$.
\end{prop}
\begin{proof}
 \begin{align*} c_j \left| \frac{\chi_{\lambda^1}(\alpha)}{\chi_{\lambda^{2,1}}(\alpha)} \right| \frac{d_{\lambda^1}}{d_{\lambda^{2,1}}} \left| \frac{\frac{\chi_{\lambda^1}(\tau)}{d_{\lambda^1}}}{\frac{\chi_{\lambda^{2,1}}(\tau)}{d_{\lambda^{2,1}}}} \right|^t &\geq \frac{1}{4} \frac{1}{{n-1 \choose 2}} \frac{n-1}{{n-1 \choose 2}}\left(1 + \frac{2}{n-1}\right)^t \\
& \geq e^{-\log((n-1)(n-2)^2) + t\frac{1}{n-1}} \\
& \geq e^{-3\log(n) + 4\log(n)} \\
& \geq 1 \\
\end{align*}

\begin{align*} c_j \left| \frac{\chi_{\lambda^1}(\alpha)}{\chi_{\lambda^{2}}(\alpha)} \right| \frac{d_{\lambda^1}}{d_{\lambda^{2}}} \left| \frac{\frac{\chi_{\lambda^1}(\tau)}{d_{\lambda^1}}}{\frac{\chi_{\lambda^{2}}(\tau)}{d_{\lambda^{2}}}} \right|^t &\geq \frac{1}{4} \frac{1}{\frac{n(n-3)}{2}} \frac{n-1}{\frac{n(n-3)}{2}}\left(1 + \frac{2}{n-1} - \frac{2}{{n \choose 2}}\right)^t \\
& \geq e^{-\log((n-3)(n)^2) + t(\frac{n-2}{n(n-1)})} \\
& \geq e^{-3\log(n) + 4\log(n)(1 - \frac{1}{n})} \\
& \geq 1 
\end{align*}

\end{proof}

\begin{prop}
$\rho$ as above will dominate all terms for partitions $\lambda = [n-j,...]$ $j \leq \frac{n}{2}$ for $t \geq 6n\log(n)+30\log(n)$
\end{prop}
\begin{proof}
Bound $\sum_{\lambda: \lambda_1 = n-j} \chi_{\lambda}(\alpha)d_\lambda$ by $\sum_{\lambda: \lambda_1=n-j} d_\lambda^2 \leq {n \choose j}^2j!$.
\begin{align*} c_j \left| \frac{\chi_{\rho}(\alpha)}{\chi_{\lambda}(\alpha)} \right| \frac{d_{\rho}}{d_{\lambda}} \left| \frac{\frac{\chi_{\rho}(\tau)}{d_{\rho}}}{\frac{\chi_\lambda(\tau)}{d_\lambda}} \right|^t  & \geq  2^{-j} \frac{{j!}}{(n)_j^2}\left(1 + \frac{(j-2)(n-j-1) -2}{{n \choose 2}}\right)^{t} \\
& \geq e^{-j\log2 -2j\log(n/\sqrt{j}) + t\frac{(j-2)(n/2-1) -2}{n(n-1)}} \\
& \geq e^{- 2j\log(n) + 2j\log(n)\left(\frac{1}{2}\frac{j-2}{j}\frac{n-j-1}{\log(n)}\right)\left(1 - \frac{2}{(j-2)(n-j-1)}\right)\frac{t}{n(n-1)}} \\
& \geq e^{2j\log(n)\left(\left(\frac{1}{2}\left(1- \frac{2}{j}\right)\frac{n-j-1}{n}\right)\left(1 - \frac{2}{(j-2)(n-j-1)}\right)\frac{t}{\log(n)(n-1)}-1\right)} \\
& \geq 1
\end{align*}
Since when $t \geq 6n\log(n)+30\log(n)$, 
\begin{align*} &\left(\left(\frac{1}{2}\left(1- \frac{2}{j}\right)\frac{n-j-1}{n}\right)\left(1 - \frac{2}{(j-2)(n-j-1)}\right)\frac{t}{\log(n)(n-1)}-1\right) \\
&\geq \left(3 + \frac{15}{n}\right)\left(\left(1- \frac{2}{j-2}\right)\frac{n-j-1}{n}\right)\left(1 - \frac{2}{(j-2)(n-j-1)}\right) \end{align*}
And this is minimized when $j=3$ giving \[ \left(3 + \frac{15}{n}\right)\frac{1}{3}\left(1 - \frac{3}{n-2}\right)\left(1 - \frac{2}{n-2}\right) \geq 0 \]

\end{proof}

\begin{prop}
$\rho$ as above will dominate all terms for partitions $\lambda = [n-j,...,]$ for $j \geq \frac{n}{2}$ for $t \geq 4n\log(n) + 2\log(2)n +20\log(n) + 50\log(2)$
\end{prop}
\begin{proof}
The character ratio difference is of constant order in this case, so it becomes very simple.

\[\frac{\chi_{\lambda{2,1}}(\tau)}{d_{\lambda^{2,1}}} - \frac{\chi_{\lambda}(\tau)}{d_{\lambda}} \geq j\left(\frac{n}{2}-1\right) -4n \geq j\left(\frac{n}{2} - 3\right) \]

\begin{align*} c_j \left| \frac{\chi_{\lambda^i}(\alpha)}{\chi_{\lambda}(\alpha)} \right| \frac{d_{\lambda^i}}{d_{\lambda}} \left| \frac{\frac{\chi_{\lambda^i}(\tau)}{d_{\lambda^i}}}{\frac{\chi_\lambda(\tau)}{d_\lambda}} \right|^t  & \geq  e^{-j\log2 -2j\log(n) + t\frac{j(\frac{n}{2} -3)}{n(n-1})} \\
& \geq e^{j\left(-\left(2\log(n) + \log(2)\right) + 2n\log(n) \frac{(n-6)(n-1)}{(n-6)(n-1)} + \log(2)n\frac{(n-1)(n-6)}{(n-6)(n-1)}\right)} \\
&\geq 1 \end{align*}

\end{proof}

This means that around $6n\log(n)$ time (and likely earlier due to cancelation of terms), permutations with more than one fixed point are above stationarity, those with one fixed point and at least two $2$-cycles are somewhere around stationarity, and those with at one fixed point and $2$-cycle or no fixed points are below uniform.  This fits well with the lower bound for total variation of this walk being established using that fixed point free permutations are too far below uniform for $t < \frac{1}{2}n\log(n)$ ~\cite{DS}.

\section{Three Cycle Walk}\label{three}

The adaptation of the above methodology to the $3$-cycle walk will give similar results.  Notably, as opposed to the transposition walk, which has lead $i$-cycle detector $\lambda^i = [n-i,i]$, this walk has lead $i$-cycle detector $\rho^i = [n-i,1...,1]$.  As a result,  $(-1)^{i+1}CL$ will be the likelihood order.  The smallest possible difference between character ratios at the $3$-cycle being $O(1/n^3)$ (as compared to $O(1/n^2)$ for the transposition) will lead to a $O(n^3)$ upper bound for when the order holds.  

\subsection{Character Ratio Maximizing $i$-Cycle Detector}

From ~\cite{Ingram} $\frac{\chi_\lambda(3)}{d_\lambda} = \frac{M_3}{2(n)_3} - \frac{3}{3(n-2)}$ where \[M_3 = \sum_{j=1}^r(\lambda_j -j)(\lambda_j -j+1)(2\lambda_j - 2j +1) + j(j-1)(2j-1) \]  Each of these products are reminicent of the formula for a sum of squares \[ \sum_{l=1}^n l^2= \frac{1}{6} n(n+1)(2n+1)\] which leads to a combinatorial interpretation for the character ratio.  This will be written most simply in Frobenius notation $\lambda = (\alpha_1,...,\alpha_k | \beta_1,...,\beta_k)$ where the main diagonal of $\lambda$ consists of $k$ boxes $(1,1)$ through $(k,k)$ and for $1 \leq j \leq k$, $a_j = \lambda_j-j$, $b_j = \lambda'_j - j$.  

\begin{align*} M_3 &= \sum_{j=1}^r(\lambda_j -j)(\lambda_j-j+1)(2\lambda_j - 2j +1) + j(j-1)(2j-1) \\
&= \left[ \sum_{j=1}^k (\lambda_j -j)(\lambda_j -j +1)(2(\lambda_j -j) + 1) \right]+ \left[\sum_{j=1}^k (j-1)((j-1)+1)(2(j-1) + 1) \right]+ \\
&\left[\sum_{j=k+1}^r \left( (j-1)((j-1)+1)(2(j-1) + 1) - (j- \lambda_j)(j-\lambda_j-1)(2(j- \lambda_j) -1) \right) \right] \\
&= \sum_{j=1}^k (\alpha_j)(\alpha_j +1)(2\alpha_j +1) + \sum_{j=1}^k 6\sum_{l=1}^{j-1} l^2 + \sum_{j=k+1}^r 6\left(\sum_{l=1}^{j-1} l^2 - \sum_{l=1}^{j-\lambda_j-1} l^2 \right) \\
& = \sum_{j=1}^k (\alpha_j)(\alpha_j +1)(2\alpha_j +1) + \sum_{j=1}^k 6\sum_{l=1}^{j-1} l^2 + \sum_{j=k+1}^r 6\left(\sum_{l=j - \lambda_j}^{j-1} l^2 \right) \\
&= \sum_{j=1}^k (\alpha_j)(\alpha_j +1)(2\alpha_j +1) + \sum_{(i,j): j>i} 6(j-i)^2 \\
& = \sum_{j=1}^k (\alpha_j)(\alpha_j+1)(2\alpha_j+1) +  \sum_{j=1}^k\sum_{l=1}^{\beta_j} 6(l)^2\\
&= \sum_{j=1}^k \left( (\alpha_j)(\alpha_j+1)(2\alpha_j+1) + (\beta_j)(\beta_j+1)(2\beta_j+1) \right) \end{align*}

So the formula for the character ratio becomes: \[ \frac{\chi_\lambda(3)}{d_\lambda} = \frac{ \sum_{j=1}^k \left( (\alpha_j)(\alpha_j+1)(2\alpha_j+1) + (\beta_j)(\beta_j+1)(2\beta_j+1) \right)}{2(n)_3} - \frac{3/2}{(n-2)} \]  Where $M_3$ can be interpreted as six times the sum over boxes in the Ferrers diagram of the square of the distance from the box to a diagonal box in its row or column.  The next step is to find the $i$-cycle detector with largest character ratio.  This has two steps, showing $M_3$ is maximized amoung $i$-cycle detectors by $\rho^i = [n-i,1,...,1]$ and $\frac{\chi_{\rho^i}(3)}{d_{\rho^i}} \geq \frac{ 3/2}{n-2}$, so a partition minimizing the always positive $M_3$ cannot maximize the character ratio.

\begin{prop}
$\rho^i$ has the maximal $M_3$ over all $i$-cycle detectors for $i \leq \frac{n-1}{2}$ with $\lambda_1 \geq \lambda_1'$.
\end{prop}
\begin{proof}
As $M_3$ is clearly equal for $\lambda$ and $\lambda'$, its enough to examine $\lambda$ with $\lambda_1 \geq \lambda_1'$.  The first row and column are clearly the $M_3$ maximizing locations for boxes as these are farthest from a diagonal.  For an $i$-cycle detector, at least $i$ boxes must be below the first row.  With $\lambda_1 \geq \lambda_1'$, this means the first row must have $i+1$ boxes.  Then moving a box from the first row at distance $(n-i-1)$ to a diagonal to the first column at distance $i$ to a diagonal never increases the distance when $i \leq \frac{n-1}{2}$.  This means $\rho^i$ and $\rho^{n-i-1}$ maximize $M_3$.  
\end{proof}

\begin{prop}
$\frac{\chi_{\rho^i}(3)}{d_{\rho^i}} \geq \frac{ 3/2}{n-2}$ for $n \geq 5$.
\end{prop}
\begin{proof}
The contribution to $M_3$ fom the $i$ blocks in the tail is $6(1^2+ \cdots + i^2)$ since $\rho^i = (n-i,i)$.
\begin{align*} \frac{\chi_{\rho^i}(3)}{d_{\rho^i}} &= \frac{(n-i-1)(n-i)(2(i-1)+1) + (i)(i+1)(2i+1)}{2(n)_3} - \frac{3/2}{n-2} \\
& = 1 - \frac{3i(n-i-1)}{(n-1)(n-2)} \end{align*}

This is minimized when $i$ is maximized at $\frac{n-1}{2}$ where the character ratio is $1 - \frac{3(\frac{n-1}{2})^2}{(n-1)(n-2)} = \frac{1}{4} - \frac{3/4}{n-2}$.  This is larger than $\frac{3/2}{n-2}$ for $n \geq 5$.  And indeed, for $n=4$, $\rho^1$ has character ratio $0$ and $\lambda^2$ is both the lead $1$ and $2$-cycle detectors. 
\end{proof}

\subsection{Variant on Cycle Lexicographical Order}

The leading terms in difference in probabilities of two elements are then the $\rho^i$, $\rho^{n-i-1}$ terms.  The relative probabilities are determined after sufficient time by the sign of the sum of these two terms.  It has been proved that for both partitions the character ratio at the $3$-cycle are positive, so the sign of the two terms is determined entirely by \[ (\chi_{\rho^i}(\alpha) - \chi_{\rho^i}(\beta)) + (\chi_{\rho^{n-i-1}}(\alpha) - \chi_{\rho^{n-i-1}}(\beta))\].  The latter part serves to generate the parity condition that only even permuations are possible.  And $\rho^i = \lambda^{i,i-1}$, which the character differences have been previously found to be $(-1)^{i+1}(a_i - b_i)$ in Proposition \ref{ik}.  So the sign of the sum of the terms is $(-1)^{i+1}$, resulting in the order $(-1)^{i+1}CL$ restricted to even permuations.

\begin{prop}
After sufficient time, the most likely element is the identity.  The least likely is either depending on $n$ a $2-,3-,4-,$ or $5-$cycle followed by the maximal number of $2$-cycles.
\end{prop} 
\begin{proof}
The most likely element is the first in $(-1)^{i+1}$CL order restricted to even permutations, the least likely the last in $(-1)^{i+1}$CL order restricted to even permutations.  See Proposition \ref{CL orders}. 

\end{proof}

\subsection{Bound for Sufficient Time}

Many of the bounds from the transposition walk can be reused to briefly show that $O(n^3)$ is sufficient time for the order to hold. The argument is omitted in this version of the paper.

\section{$n$-Cycle Walk}\label{ncycle}

This walk is significantly easier than the others to analyze due to the rarity of partitions with non-zero character ratio on the $n$-cycle.  The walk mixes in two steps ~\cite{Lulov}, and the results below find the most and least likely elements after $8$ steps and show a variant of cycle lexicographic describes the relative likelihood of the conjugacy classes after $\frac{1}{2}n\log(n)$ steps.

\subsection{Character Ratio Maximizing $i$-Cycle Detector}

Explicitly since a $n$-cycle can only insert into a partition that is a single hook: 

\[\chi_{\lambda}(n) = \left\{\begin{matrix} (-1)^k & \lambda = (n-k,1,...,1) \\ 0 & \end{matrix}\right. \]
Recall the notation for $[n-k,1...,1] = \lambda^{k,k-1} = \rho^k$.  The hook length formula gives 
\[d_{\rho^k}= \frac{n!}{n(k)!(n-k-1)!} = {{n-1}\choose{k}} \]
This restricts the partitions in this section to $\rho^k$ with character ratio:
\[ \frac{\chi_{\rho^k}(n)}{d_{\rho^k}} = \frac{(-1)^k}{{{n-1}\choose{k}}} \]

The hooks that are $i$-cycle detectors are those with $h_{2,1},h_{1,2} \geq i$.  For $\rho_k$, $h_{2,1} = k$, $h_{1,2} = n-1-k$. The character ratios are symmetric up to sign about $\frac{n-1}{2}$.  Therefore the $i$-cycle detectors with largest character ratio at the $n$-cycle are $\rho^i, \rho^{n-i-1}$.  And the discrete Fourier inversion formula reduces to the equation:

\begin{prop}
Let $\alpha,\beta$ be conjugacy classes with the same number of $j$ cycles for $j<i$, then
\[P^{*t}(\alpha) - P^{*t}(\beta) = \frac{1}{n!} \sum\limits_{k=i}^{n-i-1} [\chi_{\rho^k}(\alpha) - \chi_{\rho^k}(\beta)] {{n-1}\choose{k}} \left(\frac{(-1)^k}{{{n-1} \choose {k}}}\right)^t \]
\end{prop}

\subsection{Variation on Cycle Lexicographical Order}

The order of likelihoods of conjugacy classes is determined by the sign of the terms for $k=i,n-i-1$. The computation is now complicated by the sign of the character ratios.

\begin{prop}
The sign of the sum of the $\rho^i,\rho^{n-i+1}$ terms for the $n$-cycle walk is $(-1)^{i(t+1)+1}$.  Yielding $-$CL order at odd times and $(-1)^{i+1}$CL at even times.
\end{prop}
\begin{proof}
Assume $\alpha$,$\beta$ are two conjugacy classes of the same sign with $a_i>b_i$ and $a_j = b_j$ for $j<i$.

\begin{align*} &[\chi_{\rho^i}(\alpha) - \chi_{\rho^i}(\beta)] {{n-1}\choose{i}} \left(\frac{(-1)^i}{{{n-1} \choose {i}}}\right)^t + [\chi_{\rho^{n-i-1}}(\alpha) - \chi_{\rho^{n-i-1}}(\beta)] {{n-1}\choose{i}} \left(\frac{(-1)^{n-i-1}}{{{n-1} \choose {i}}}\right)^t \\
&= \left([\chi_{\rho^i}(\alpha) - \chi_{\rho^i}(\beta)](-1)^{it} + [\chi_{\rho^{n-i-1}}(\alpha) - \chi_{\rho^{n-i-1}}(\beta)](-1)^{(n-i-1)t}\right) {{n-1} \choose i}^{1-t} \\
& = \left((-1)^{i+1}(-1)^{it} + (-1)^{i+1}sgn(\alpha)(-1)^{(n-i-1)t}\right) (a_i - b_i){{n-1} \choose i}^{1-t} \\
&= (-1)^{i(t+1)+1}\left(1 + sgn(\alpha)(-1)^{(n-1)t}\right) (a_i - b_i){{n-1} \choose i}^{1-t} \end{align*}

The sign of this expression is $(-1)^{i(t+1)+1}$, while the $(1 + sgn(\alpha)(-1)^{(n-1)t})$ term just determines that the walk is restricted to either odd or even conjugacy classes at various times.  When $t$ is odd, $(-1)^{i(t+1)+1} = -1$ and after sufficient time $\alpha$ will be less likely than $\beta$ if $a_i > b_i$ with $a_j = b_j$ for $j<i$.  When $t$ is even, $(-1)^{i(t+1)+1}= (-1)^{i+1}$ and after sufficient time $\alpha$ will be more(less) likely than $\beta$ if $i$ is odd(even) and $a_i > b_i$ with $a_j = b_j$ for $j<i$
\end{proof}

\begin{cor}
For $t$ sufficiently larger the order is as follows.  When $t$ is odd, the most likely element is the $n$-cycle.  When $t$ is even, its the identity.  When $t$ is odd, $n$ even, the least likely is the transposition, while for $n$ odd it is the identity.  For $t$ even the least likely element is a bit more complicated.  It is whichever of a $2$-,$3$-,$4$-,or $5$-cycle and the rest $2$ cycles is fixed point free and the appropriate sign (respectively when $4|n,4|n+1,4|n+2,4|n+3$).
\end{cor}
\begin{proof}

At odd times, after sufficient time, the most likely element will be the first in -CL order, the least likely the last in -CL order of the same parity as $n-1$.  At even times, after sufficient time, the most likely is the first even element in $(-1)^{i+1}$CL order, the least likely the last even element in the same order. See Proposition \ref{CL orders} for a proof of what these are.    
\end{proof}

\subsection{Bound for Sufficient Time}
As before, the method is to bound the ratio of each term in the sum to the lead $i$-cycle detector.  In this case, there are two pieces to analyze since the character ratio and degree of the representation terms naturally combine. Five conclusions will be made at different time points as more of the order can be shown to hold.

\begin{prop}
Let $\alpha, \beta$ be conjugacy classes with $a_k = b_k$ for $k<i$, $a_i \neq b_i$.  
For $i = 1<j$,
\[ \left| \frac{\chi_{\rho^{i}}(\alpha) - \chi_{\rho^{i}}(\beta)}{\chi_{\rho^{j}}(\alpha) - \chi_{\rho^{j}}(\beta)} \right| \geq \left({n-1 \choose j}\right)^{-1} \]
For $2 \leq i < j$,
\[ \left| \frac{\chi_{\rho^{i}}(\alpha) - \chi_{\rho^{i}}(\beta)}{\chi_{\rho^{j}}(\alpha) - \chi_{\rho^{j}}(\beta)} \right| \geq \left({n-2i \choose j-i}\frac{(n-i)(j-i)}{i} \right)^{-1} \]
For $2 \leq i < j$, $a_k = b_k = 0$ for $k<i$,
\[ \left| \frac{\chi_{\rho^{i}}(\alpha) - \chi_{\rho^{i}}(\beta)}{\chi_{\rho^{j}}(\alpha) - \chi_{\rho^{j}}(\beta)} \right| \geq \left( (j-i){\frac{n-i}{i} \choose \frac{j}{i}} \right)^{-1} \]

\end{prop}
\begin{proof}
The calculation here is simplified from the general case since we only consider such a restricted case of partitions. In both cases, the top is bounded from below as \[\left| \chi_{\rho^{i}}(\alpha) - \chi_{\rho^{i}}(\beta) \right| = \left|(-1)^{k+1}(a_i - b_i) \right| \geq 1\].  To bound the denominator in the case $i=1$, the worst case for each term is taken: \[\left| \chi_{\rho^{j}}(\alpha) - \chi_{\rho^{j}}(\beta) \right| \leq 2d_{\rho_j}  = 2{ n-1 \choose j}\] 

When $i \geq 2$, the cancellation in the denominator resulting from trivial insertions of $i$ and larger cycles allows a slightly improved bound.  It has been assumed that $\alpha, \beta$ have identical cycle structure in cycles smaller than $i$. After trivially inserting all the $i$ and larger cycles of $\alpha,\beta$ into $\rho_j$, the remaining pieces are identical for both conjugacy classes, and so the values resulting from this cancel.  It remains to count the number of ways of inserting with at least one $i$ or larger cycle inserting non-trivially into the lower portion of $\rho_j$.    And these are chosen from at most $n-2i$ cycles remaining after the first non-trivial insertion. Then there are at most $j-i$ rows to start the trivial insertion and at most $\frac{n-i}{i}$ choices for cycle the first non-trivial insertion.

Now consider the case where $a_k=b_k=0$ for $k<i$. There are at most $j-i$ rows in which to start the trivial insertion. Under these conditions, at most $j-i$ cycles of size $<1$ and at most $\frac{j}{i}$ cycles of size $\geq i$ can fit below the first line in $\lambda_j$. And of course at most $\frac{n-i}{i}$ cycles larger than $i$ can be left after the nontrivial insertion. And inserting into $\rho^j$ there is only one way to position them. 
\end{proof}

\begin{prop}
For $i < j\leq \frac{n-1}{2}$
\[ \left| \frac{\chi_{\rho^{i}}}{d_{\rho^{i}}} \left(\frac{\chi_{\rho^{j}}}{d_{\rho^{j}}} \right)^{_1}\right|^{t-1} = \left(\frac{{n-1 \choose j}}{{n-1 \choose i}} \right)^{t-1} \geq \left( \frac{n-i-1}{j} \right)^{(j-i)(t-1)} \] 
\end{prop}
\begin{proof}
The first equality is as calculated before.  The second follows from the ratio of $\frac{n-i-1-k}{j-k}$ increasing as $k$ increases to $j-i$ when $i,j \leq \frac{n-1}{2}$.
\end{proof}

The time bounds needed for the least/most likely elements are quite a bit smaller than the bounds needed for the general order to hold. 

\begin{lem}
Let $\alpha, \beta$ be two conjugacy classes of the same parity with $a_1 > b_1$ so that $i=1$.  Then for $t \geq 4$, $n \geq 18$, the order holds and $\alpha$ is more likely at even times, $\beta$ at odd times.
\end{lem}
\begin{proof}
This will follow if the ratio of the $\rho^j$ term to the $\rho^1$ term is smaller than $2^{j-1}$ for $t \geq 4$.

\begin{align*} & 2^{1-j}\left| \frac{\chi_{\rho^{1}}(\alpha) - \chi_{\rho^{1}}(\beta)}{\chi_{\rho^{j}}(\alpha) - \chi_{\rho^{j}}(\beta)} \right|\left| \frac{\chi_{\rho^{1}}(n)}{d_{\rho^{1}}} \left(\frac{\chi_{\rho^{j}(n)}}{d_{\rho^{j}}} \right)^{-1}\right|^{t-1} \\
&\geq 2^{1-j} \left(2{n-1 \choose j}\right)^{-1} \left(\frac{{n-1 \choose j}}{{n-1 \choose 1}} \right)^{t-1} \\
& = \frac{{{n-1 \choose j}}^{t-2}}{2^{j}(n-1)^{t-1}} \\
& \geq  \frac{{{n-1 \choose j}}^{2}}{2^{j}(n-1)^{3}} \\
\end{align*}
Letting $j$ vary, the expression is increasing when $j \geq \frac{n}{3} -1$ so its enough to find $n$ sufficiently large that the expression is larger than $1$ for $j=2, \frac{n-1}{2}$.
\[ \frac{{n-1 \choose 2}^2}{2^2(n-1)^3} = \frac{(n-2)^2}{16(n-1)} = \frac{n-2}{16}\left( 1- \frac{1}{n-1} \right) \geq 1 \text{ when } n\geq 18 \]
\[ \frac{{n-1 \choose \frac{n-1}{2}}^2}{2^{\frac{n-1}{2}}(n-1)^3} = \frac{(n-1)!!}{\frac{n-2}{2}!}\frac{{n-1 \choose \frac{n-2}{2}}}{(n-1)^3} \geq 1 \text{ when } n \geq 18 \]
\end{proof}

\begin{lem}
Let $\alpha, \beta$ be two conjugacy classes of the same parity with $a_j = b_j = 0$ for $j < i$, $a_i > b_i$.  If $2\leq i \leq n/3$ the order holds for $t \geq 8$.  If $i > n/3$ the order holds at all times.
\end{lem}
\begin{proof}
The only conjugacy classes of this form with $i > n/3$ are partitions with smallest piece of size greater than $n/3$, which means they have at most two pieces.  Since $[n]$ and $[n-a,a]$ have opposite parities, these cannot be reached at the same parity of time. A partition with only two parts and smallest part $a$ can only insert in at most two ways into $\rho^k$.

 \[ \chi_{\rho^k([n-a,a])} =\left\{\begin{array}{l l} (-1)^{k} + (-1)^{a+1}(-1)^{k-a} = 0 & \textrm{ when $a \leq k \leq \frac{n-1}{2}$} \\
																							(-1)^k & \textrm{ when $a > k$}	\end{array} \right. \]

Only one of $\alpha, \beta$ can have a part of size $i$, so this must be $\alpha$.  And $\chi_{\rho^j}(\alpha) = 0$ for $j \geq i$, while $\chi_{\rho^i}(\beta) = (-1)^i$,  $\chi_{\rho^j}(\beta) = 0$ for $j > i$.  This means that the only $\rho^i$, $\rho^{n-1-i}$ terms are nonzero, hence their sign determines the order at all times.

In the case that $i \leq n/3$, recall that the bounds of the ratio of the difference of characters was significantly improved. It remains to show the following for $t \geq 5$, $i \leq n/3$, $\frac{n-1}{2} \geq j \geq i$:

\[2^{i-j}\left| [\chi_{\rho^{i}}(\alpha) - \chi_{\rho^{i}}(\beta)] {{n-1}\choose i}\left(\frac{(-1)^i}{{{n-1} \choose {i}}}\right)^t \right| > \left| [\chi_{\rho^{j}}(\alpha) - \chi_{\rho^{j}}(\beta)] {{n-1}\choose{j}} \left(\frac{(-1)^j}{{{n-1} \choose {j}}}\right)^t \right| \]

Dividing this becomes:

\[ 2^{i-j}\left|\frac{\chi_{\rho^{i}}(\alpha) - \chi_{\rho^{i}}(\beta)}{\chi_{\rho^{j}}(\alpha) - \chi_{\rho^{j}}(\beta)} \right|\left(\frac{{{n-1}\choose j}}{{{n-1}\choose i}}\right)^{t-1} \geq 1 \]

\begin{align*} & 2^{i-j}\left|\frac{\chi_{\rho^{i}}(\alpha) - \chi_{\rho^{i}}(\beta)}{\chi_{\rho^{j}}(\alpha) - \chi_{\rho^{j}}(\beta)} \right|\left(\frac{{{n-1}\choose j}}{{{n-1}\choose i}}\right)^{t-1} \\ 
&\geq 2^{i-j}(j-i)\left(\frac{(n-i-1)_{j-i}}{(j)_{j-i}}\right)^{t-2} \\
& \geq 2^{i-j}(j-i) \left(\frac{(2n/3 -1)_{j-i}}{((n-1)/2)_{j-i}} \right)^{t-2} \\
& \geq 2^{i-j}(j-i) \left(\frac{4}{3} \right)^{(j-i)(t-2)} \\
& \geq \frac{4}{3}^{(j-i)(t-2) -\log(2)(j-i) -\ln(j-i)} \\
&\geq 1 \text{ when $t \geq 6$} \end{align*}
\end{proof}

\begin{lem}
For $\alpha, \beta$ two conjugacy classes of the same parity with $a_j = b_j$ for $j < i$, $a_i > b_i$.  If $i \leq n/3$ the order holds when $t \geq \log(n) + 3$.  For $i > n/3$ the order holds when $t \geq \frac{n-1}{2}\log(n) + 3$.
\end{lem}

\begin{proof}
Below will be needed the relation that $\frac{(n-i-1)_{j-i}}{(j)_{j-i}} { n-2i \choose j-i}^{-1} \leq \left(\frac{1}{i+1} \right)^{j-i}$ Like before:

\begin{align*} & 2^{i-j}\left|\frac{\chi_{\rho^{i}}(\alpha) - \chi_{\rho^{i}}(\beta)}{\chi_{\rho^{j}}(\alpha) - \chi_{\rho^{j}}(\beta)} \right|\frac{{{n-1}\choose j}}{{{n-1}\choose i}}^{t-1} \\
& \geq  2^{i-j}\left(\frac{(n-i-1)_{j-i}}{(j)_{j-i}}\right)^{t-1} \left({n-2i \choose j-i}(j-i) \right)^{-1} \\
& \geq 2^{i-j-1}\left(\frac{(n-i-1)_{j-i}}{(j)_{j-i}} \right)^{t-1}  {n-2i \choose j-i}^{-1} \\ 
& \geq 2^{i-j}(i+1)^{i-j}\left(\frac{(n-i-1)_{j-i}}{(j)_{j-i}} \right)^{t-2}\\
& \geq 2^{i-j}(i+1)^{i-j}\left(1 + \frac{n-i-j-1}{j}\right)^{(j-i)(t-2)} \\
& \geq \left(\frac{(1 + \frac{n-i-j-1}{j})^{t-2}}{2(i+1)} \right)^{j-i}
\end{align*}

So it remains to find $t-2$ such that $(1 + \frac{n-i-j-1}{j})^{t-3} \geq 2(i+1)$.  Assuming that $i \leq n/3$:

\begin{align*} (1 + \frac{n-i-j-1}{j})^{t-2} &\geq (1+ \frac{1 - n/3 - (n-1)/2 -1}{(n-1)/2})^{t-2} \\
& \geq(1 + \frac{n/6 - 1/2}{n/2 - 1/2})^{t-3} \\
& \geq(\frac{4}{3})^{t-3} \\
& \geq 2(n/3+1) \geq 2(i+1) \text{ when $t \geq \frac{7}{2}\log(n) + 2$} \end{align*}

And only assuming that $i \leq \frac{n-3}{2}, j \leq \frac{n-1}{2}$,

\begin{align*} (1 + \frac{n-i-j-1}{j})^{t-2} & \geq ( 1 + \frac{1}{(n-1)/2})^{t-2} \\
& \geq 2(\frac{n-3}{2} + 1) \text{ when } t \geq \frac{4}{3}n \log{2(n-3)} \end{align*}

\end{proof}


\section{Other Remarks}

Many other random walks generated by a conjugacy class fail in various ways to have such nice likelihood orders after sufficient time.  There may not be a unique largest $i$-cycle detector or various values of $i$ may share a single partition.  Or indeed, the largest $i$-cycle detector may be blind to $i$-cycles as $[n-3,2,1]$ is to $2$-cycles.  When the lead $i$-cycle detector fails to differentiate conjugacy classes, it falls to the next largest $i$-cycle detector and so on.  Another issue is in adding a lazy component to a walk the likelihood order may change decidedly when for some $i$ the largest $i$-cycle detectors all have negative character ratios.  $[n-i,i],[n-i,1^i]$ are particularly nice in these regards, though the latter fails to be nice for lazy walks.  These failures can lead to partial or partially describable likelihood orders rather than total ones.  

For example, in the case of the random walk generated by $(n-1)$-cycles, the character ratio is zero except at $\lambda = [n],[1^n],[n-i,2,1^{i-2}]$.  Where $[n-i,2,1^{i}]$ is a $i$-cycle detector.  $[n-2,2]$ doubles as the lead $1$ and $2$-cycle detector where $\chi_{[n-2,2]}(\alpha) = a_2 + {a_1 \choose 2} - a_1$.  And the sign of $a_i$ in $\chi_{[n-i,2,i]}(\alpha)$ is $(-1)^i$.  The order after sufficient time is first based on $(a_2-b_2) + \frac{(a_1-b_1)(a_1+b_1-3)}{2} > 0$, but if this does not resolve the $1$,$2$-cycle detection, it falls to the $3$-cycle detector now the lead $1,2,3$-cycle detector to resolve the dispute and so on.  In the case that both conjugacy classes have no $1$ or $2$-cycles at even times it becomes just $(-1)^iCL$ order.

 \bibliography{LikelihoodOrdersSym}
 \bibliographystyle{plain}
 \end{document}